\documentclass{article}

\usepackage[utf8]{inputenc}
\usepackage[T1]{fontenc}
\usepackage{amsmath,amsfonts,amssymb}
\usepackage{amsthm,upref}
\usepackage[a4paper]{geometry}
\usepackage{color} 
\usepackage[dvipsnames]{xcolor}
\usepackage{graphicx}
\usepackage{subfigure}
\usepackage{tabularx,float}
\usepackage[hidelinks]{hyperref}
\usepackage{comment} 
\usepackage{tabularx} 
\usepackage{enumerate,mathrsfs}
\usepackage{bm} 

\usepackage{fullpage}

\newcommand{\note}[1]{ \textcolor{Maroon}{ \ [ \ #1 \  ] \ }} 
\newtheorem{thm}{Theorem}[section]

\newtheorem{proposition}[thm]{Proposition}

\newtheorem{lemma}[thm]{Lemma}
\newtheorem{remark}[thm]{Remark}

\newtheorem{hyp}{Hypothesis}

\newcommand{\eps}{\varepsilon}

\newcommand{\R}{\mathbb{R}}
\newcommand{\D}{{\mathrm D}}
\newcommand{\Lop}{{\mathrm L}}
\newcommand{\Nop}{{\mathrm N}}
\newcommand{\dd}{{\mathrm d}}
\newcommand{\transpose}{\textnormal{T}}
\newcommand{\e}{{\mathrm e}}

\usepackage{todonotes}
\newcommand\SJ[1]{{\color{black}{#1}}} 

\title{Characterising exchange of stability in scalar reaction-diffusion equations via geometric blow-up}
\author{S.~Jelbart\thanks{Vienna University of Technology, Wiedner Hauptstrasse 8, 1040 Vienna, Austria.}, C.~Kuehn\thanks{Technical University of Munich, Boltzmannstrasse 3, 85748 Garching, Germany.} \ \& A.~Martínez Sánchez\thanks{University of Barcelona, Gran Via de les Corts Catalanes 585, 08007 Barcelona, Spain. Corresponding author. Email: amartinezsanchez@ub.edu}}

\date{}

\begin{document}
	
	\maketitle
	
	\begin{abstract}
		We study the exchange of stability in scalar reaction-diffusion equations which feature a slow passage through transcritical and pitchfork type singularities in the reaction term, using a novel adaptation of the geometric blow-up method. Our results are consistent with known results on bounded spatial domains which were obtained by Butuzov, Nefedov \& Schneider using comparison principles like upper and lower solutions in \cite{Butuzov2000}, however, from a methodological point of view, the approach is motivated by the analysis of closely related ODE problems using geometric blow-up presented by Krupa \& Szmolyan in \cite{Krupa2001c}. After applying the blow-up transformation, we obtain a system of PDEs which can be studied in local coordinate charts. Importantly, the blow-up procedure resolves a spectral degeneracy in which continuous spectrum along the entire negative real axis is `pushed back' so as to create a spectral gap in the linearisation about particular steady states which arise within the so-called entry and exit charts. This makes it possible to extend slow-type invariant manifolds into and out of a neighbourhood of the singular point using center manifold theory, in a manner which is conceptually analogous to the established approach in the ODE setting. We expect that the approach can be adapted and applied to the study of dynamic bifurcations in PDEs in a wide variety of different contexts.
	\end{abstract}
	
	\bigskip
	
	
	

\section{Introduction}\label{sec:Introduction}

Due to the pioneering and seminal works of \cite{Dumortier1996,Hayes2016,Krupa2001a,Krupa2001c,Krupa2001b,Szmolyan2001,Szmolyan2004,Wechselberger2012}, the \textit{geometric blow-up method} has been established as a powerful approach to the study of the local geometry and dynamics associated with singularities of fast-slow systems of ODEs; see also \cite{Jardon2019b} for a review and \cite{deMaesschalck2021} for a recent and extensive treatment in the case of planar systems.

The primary aim of this article is to show that the geometric blow-up method can also be applied to the study of ``slow passage'' or ``dynamic bifurcation'' phenomena in scalar reaction-diffusion equations of the form
\begin{equation}
	\label{eq:main_general}
	\begin{split}
		\partial_t u &= \partial_x^2 u + f(u,\mu,\eps) , \\
		\dot \mu &= \eps ,
	\end{split}
\end{equation}
where $t > 0$, $x \in \R$, $0 < \eps \ll 1$, $u = u(x,t) \in \R$ is the unknown, $\mu = \mu(t) = \mu(0) + \eps t \in \R$ is a dynamic bifurcation parameter/slow variable, and the overdot denotes differentiation with respect to $t$. We shall consider two different possibilities for the reaction dynamics, i.e.~for the function $f$, which correspond to cases in which the reaction dynamics when $\eps = 0$ feature either a transcritical or pitchfork bifurcation at $u = \mu = 0$. The corresponding slow passage problem is to understand the fate of solutions with initial conditions satisfying $\mu(0) < 0$ as they drift through a neighbourhood of the (static) bifurcation point in time, according to $\dot \mu = \eps$, when $0 < \eps \ll 1$.

The particular slow passage problems outlined above have been rigorously analysed in the case of a bounded spatial domain in \cite{Butuzov2000}. In both the transcritical and pitchfork cases, the authors were able to prove, under suitable genericity conditions, that the solutions of interest undergo an exchange of stability phenomenon. In particular, it was shown that these solutions are exponentially close to a branch of spatially homogeneous steady states which is determined by the reaction dynamics when they emerge from a neighbourhood of the singularity with a small but fixed value of $\mu > 0$. They provided a succinct and elegant proof of this fact using a combination of singular perturbation arguments and -- crucially -- suitable upper and lower solutions. Indeed, approaches based on upper and lower solutions have been applied in order to derive rigorous statements on exchange of stability, delayed stability loss and ``canard phenomena'' in a range of slow passage problems in the general form \eqref{eq:main_general}; we refer to \cite{Butuzov2002b,Butuzov1999,Butuzov2001,Butuzov2002c,Maesschalck2009}. More generally, we refer to \cite{Asch2024,Goh2024,Goh2023,Goh2022,Kaper2018} for recent work on slow passage phenomena in PDEs, and to \cite{Avitabile2020} for a rigorous geometric approach which applies to problems which exhibit a spectral gap property at the singularity.

While arguments based on the identification of suitable upper and lower solutions can lead to elegant, succinct and rigorous results (as is the case in many of the works mentioned above), there are well-known methodological disadvantages to such approaches. In particular, comparison methods of this kind tend to become untenable in multi-component systems, and even for scalar equations, there is no systematic way to identify the `correct' upper and lower solutions. In the ODE setting, the geometric blow-up method offers an established and far more constructive alternative, which can be adapted and applied in a wide variety of different contexts. In the particular case of system \eqref{eq:main_general} with either transcritical- or pitchfork-type reaction kinetics, the spatially homogenous solutions of system \eqref{eq:main_general} are described by a planar, fast-slow system of ODEs, and the local geometry and dynamics associated with this problem \SJ{have} been described in detail using geometric blow-up techniques in \cite{Krupa2001c}. \SJ{Although we shall primarily be interested in approaches based on geometric blow-up of the kind presented in \cite{Krupa2001c} in this work, it is worthy to note that local asymptotic analyses of the dynamic transcritical and pitchfork bifurcations in the ODE setting can be traced back to the work of \cite{Haberman1979} at least; we refer also to \cite{Benoit1991} and the many references therein for approaches and advances in the study of dynamic bifurcations in ODEs more generally, which predate the aforementioned works that are based on geometric blow-up techniques.} 

In this article, we provide a rigorous and geometric description of exchange of stability phenomena which mirror the results for bounded domains in \cite{Butuzov2000} (except now on an unbounded domain), using a geometric blow-up based methodology which mirrors the proof of the corresponding results on the ODE variant of the problem in \cite{Krupa2001c}. We begin by deriving suitable local normal forms for both the transcritical and pitchfork problems\SJ{, which we then proceed to study in detail. The} main technical challenge stems from the fact that the geometric blow-up method is still being developed for PDEs. Our approach is based on recent work in \cite{Jelbart2022,Jelbart2023}, where the authors developed and applied geometric blow-up techniques in the context of scalar slow passage problems which are posed on an unbounded domain. On this approach, the spatial variable $x$ is incorporated in the blow-up itself, via a state-dependent transformation which allows for the characteristic spatial scale to vary in time as solutions traverse a neighbourhood of the singularity. For an alternative approach which has been applied to spatially discretized problems on bounded domains which feature more complicated slow dynamics, we refer to \cite{Engel2021,Engel2024,Engel2020,Zacharis2023}.

Crucially, the blow-up procedure allows us to resolve a spectral degeneracy which occurs in the static system \eqref{eq:main_general}$|_{\eps = 0}$ when $\mu = 0$, where the (purely continuous) essential spectrum associated with the linearisation about \SJ{the trivial} steady state \SJ{$u = 0$} lies along the entire negative real axis of the complex plane. Following geometric blow-up and desingularization, the linearised equations in the phase-directional (entry/exit) charts exhibit a spectral gap property, whereby the essential spectrum is pushed into the left-half plane, and two point eigenvalues appear at the origin. We interpret this as a direct generalisation of the key mathematical advantage of the blow-up method in the ODE setting, namely, that it can be used to resolve the loss of hyperbolicity at a degenerate fixed point by mapping the problem to a blown-up space in which the dynamics are governed by a desingularised vector field that it is either hyperbolic or partially hyperbolic when evaluated at its singularities. 

The spectral gap property allows us to analyse the dynamics in the phase-directional charts using center manifold theory. This is directly analogous to the approach taken in \cite{Krupa2001c}, which is typical in problems of this kind, except that we work with PDEs and therefore require a variant of center manifold theory which applies in Banach spaces. The variant that we use is based on \cite{Vanderbauwhede1992}, in the well-known formulation that has been presented and extended in \cite{Haragus2010}. Using this approach, we identify a number of $2$-dimensional local center manifolds which are comprised entirely of spatially homogeneous solutions. Indeed, these center manifolds can be identified with the center manifolds which appear in the phase-directional charts of the ODE analysis in \cite{Krupa2001c} (where they can be viewed as extensions of Fenichel slow manifolds), except that in our case, they are embedded and strongly attracting in a suitable \SJ{(infinite-dimensional)} Banach space. These observations imply that the extension of these manifolds into the so-called family rescaling chart, where the dynamics in an $\eps$-dependent neighbourhood of the origin are described, is determined by the ODE analysis in \cite{Krupa2001c}. This allows us to connect the extended center manifolds, however, it must still be shown that the PDE solutions of interest actually `track' this connection. In order to do so, we show that a priori estimates on the size of the separation between solutions and the connecting extended center manifold can be controlled using a direct but relatively constructive approach which is conceptually similar to the established approach to obtaining finite time approximation results for the validity of modulation equations as in, e.g.~\cite{Collet1990,Kirrmann1992,Kuehn2019,Schneider2017}.

More generally, the real power of the approach developed herein is that the methodology has a clear parallel to the established approach to the local analysis of dynamic or slow-fast singularities in ODE systems using geometric blow-up. Based on this fact, we are hopeful that the methods developed herein can -- as they do in the ODE setting -- be transferred and adapted in order to obtain rigorous and geometrically informative information about the dynamics close to a wide variety of dynamic bifurcations occurring in a wide variety of PDEs.

The paper is organised as follows: In Section \ref{sec:Classical_bifurcation_in_the_singular_limit} we introduce the problem, derive the local normal forms with which we work in later sections, and identify the bifurcation in the static singular limit problem \eqref{eq:main_general}$|_{\eps = 0}$. In Section \ref{sec:main_results} we state and describe our main results, which we then prove using geometric blow-up in Section \ref{sec:proofs}. Finally in Section \ref{sec:summary_and_outlook}, we summarise and discuss our findings. The Appendix contains the list of hypotheses from \cite{Haragus2010} which need to be verified in order to use the center manifold theorem which is crucial for our analysis, as well as some functional-analytic properties of the spaces that we use for analysis and some direct estimates that are useful for the proofs.

\section{Setup and the singular limit}\label{sec:Classical_bifurcation_in_the_singular_limit}

As noted in Section \ref{sec:Introduction}, we consider slow passage problems in scalar reaction-diffusion systems of the form \eqref{eq:main_general}, for which the reaction function $f$ is of local transcritical or pitchfork type. More precisely, we write \SJ{$\mathcal Q := (0,0,0)$} and assume that $f$ satisfies the conditions in either (T) and (P) below:
\begin{enumerate}
	\item[(T)] The reaction term $f$ has a \textit{transcritical singularity} at the origin, i.e.~
	\[
	f(\mathcal Q) = 0, \qquad 
	\partial_u f(\mathcal Q) = 0, \qquad 
	\partial_\mu f(\mathcal Q) = 0,
	\]
	and
	\[
	\det
	\begin{pmatrix}
		\partial_{u}^2 f (\mathcal Q) & 
		\partial_{u \mu}^2 f (\mathcal Q) \\
		\partial_{\mu u}^2 f (\mathcal Q) &
		\partial_{\mu}^2 f (\mathcal Q)
	\end{pmatrix} < 0,  \qquad
	\partial_{u}^2 f (\mathcal Q) \neq 0 ,
	\]
	as in Figure \ref{fig:critical_manifolds}(a).
	\item[(P)] The reaction term $f$ has a \textit{pitchfork singularity} at the origin, i.e.~
	\[
	f(\mathcal Q) = 0, \qquad
	\partial_u f (\mathcal Q) = 0, \qquad
	\partial_\mu f (\mathcal Q) = 0, \qquad
	\partial_u^2 f (\mathcal Q) = 0, 
	\]
	and
	\[
	\partial_u^3 f (\mathcal Q) \neq 0, \qquad
	\partial_{u \mu}^2 (\mathcal Q) \neq 0,
	\]
	as in Figure \ref{fig:critical_manifolds}(b).
\end{enumerate}

\begin{remark}
	The defining conditions in (T) and (P) are the same as those appearing in \cite{Krupa2001c}, except that the `slow regularity condition' which ensures a slow passage through the singularity is missing. This condition is automatically satisfied in our case due to the simple form of the slow equation $\dot \mu = \eps$. 
\end{remark}

\begin{figure}
	\centering
	\subfigure[Transcritical case.]{
		\includegraphics[width=0.4\textwidth]{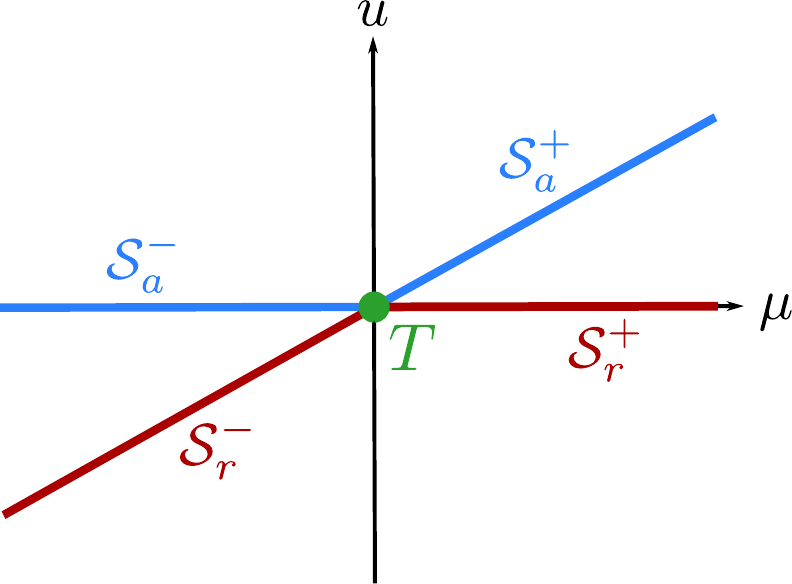}} \qquad
	\subfigure[Pitchfork case.]{
		\includegraphics[width=0.4\textwidth]{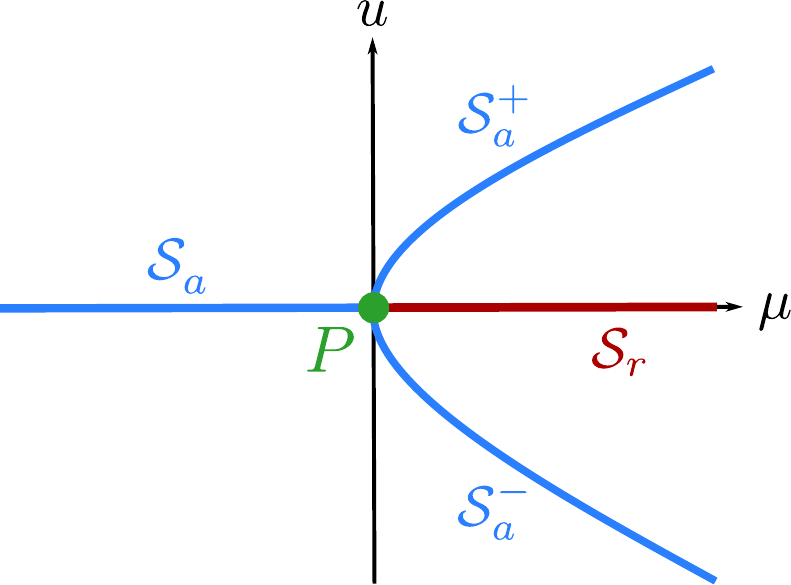}}
		\caption{Local normal form geometry of the zero sets obtained by solving $f(u,\mu,0) = 0$ in (a) the transcritical case (T), and (b) the pitchfork case (P). Considered as critical manifolds of the planar system of fast-slow ODEs in \eqref{eq:ODEs}, the branches $\mathcal S_a$ and $\mathcal S_a^\pm$ ($\mathcal S_r$ and $\mathcal S_r^\pm$) shown in blue (red) are normally hyperbolic and attracting (repelling), while the origin (green disk) is non-normally hyperbolic and denoted by $T$ and $P$ in the transcritcal and pitchfork cases shown in (a) and (b) respectively.}
		\label{fig:critical_manifolds}
\end{figure}

For analytical purposes, it is preferable to work with suitable local normal forms. In order to simplify the notation appearing in the derivation of these systems, we write
\begin{equation}
	\label{eq:coefficients_transcritical}
		\alpha = \frac{1}{2} 
		\partial_u^2 f (\mathcal Q) , \qquad 
		\beta = \frac{1}{2} 
		\partial_{u \mu}^2 f (\mathcal Q) , \qquad
		\gamma = \frac{1}{2} 
		\partial_\mu^2 f (\mathcal Q) , \qquad 
		\delta = 
		\partial_\eps f (\mathcal Q) , 
\end{equation}
and
\begin{equation}
	\label{eq:coefficients_pitchfork}
		\tilde \alpha = 
		\partial_{u \mu}^2 f (\mathcal Q) , \qquad 
		\tilde \beta = \frac{1}{2} 
		\partial_\mu^2 f (\mathcal Q) , \qquad
		\tilde \gamma = \frac{1}{2} 
		\partial_u^3 f(\mathcal Q) , \qquad 
		\tilde \delta = 
		\partial_\eps f (\mathcal Q) ,
\end{equation}
c.f.~the notation introduced just prior to the local normal form lemmas appearing in \cite{Krupa2001c}. We obtain the following result.

\begin{lemma}
	\label{lem:normal_forms}
	The following assertions hold for system \eqref{eq:main_general}:
	\begin{enumerate}
		\item Assume that $f$ satisfies the transcritical conditions in (T), with $\alpha > 0$ and $\mathcal D := \sqrt{\beta^2 - \alpha \gamma} > 0$. Then there is a change of coordinates which transforms system \eqref{eq:main_general} into
		\begin{equation}
			\label{eq:main_transcritical}
			\begin{split}
				\partial_t u &= \partial_x^2 u + \mu u - u^2 + \eps \lambda + \mathcal R(u, \mu, \eps) , \\
				\dot \mu &= \eps ,
			\end{split}
		\end{equation}
		where
		\begin{equation}
			\label{eq:lambda_transcritical}
			\lambda = \frac{(\alpha \delta + \beta - \mathcal D)}{2 \alpha \mathcal D} 
		\end{equation}
		and $\mathcal R(u, \mu, \eps) = \mathcal O(u^3, u^2 \mu, u \mu^2, u \eps, \mu \eps, \eps^2)$.
		\item Assume that $f$ satisfies the pitchfork conditions in (P), with $\tilde \alpha > 0$ and $\tilde \gamma < 0$. Then there is a change of coordinates which transforms system \eqref{eq:main_general} into
		\begin{equation}
			\label{eq:main_pitchfork}
			\begin{split}
				\partial_t u &= \partial_x^2 u + \mu u - u^3 + \eps \lambda + \mathcal R(u, \mu, \eps) , \\
				\dot \mu &= \eps ,
			\end{split}
		\end{equation}
		where
		\begin{equation}
			\label{eq:lambda_pitchfork}
			\lambda = \frac{(\tilde \alpha \tilde \delta - \tilde \beta) \sqrt{- \tilde \gamma}}{\tilde \alpha^2} 
		\end{equation}
		and $\mathcal R(u, \mu, \eps) = \mathcal O(u^4, u^2 \mu, u \mu^2, u \eps, \mu \eps, \eps^2)$.
	\end{enumerate}
\end{lemma}

\begin{proof}
	The proof is similar to the proof of to Lemma's 2.1 and 4.1 in \cite{Krupa2001c}, except that we perform an additional rotation in the transcritical case, and (in both cases) an additional near-identity transformation which rectifies one of the branches of steady states along $u = 0$.
	
	Explicitly, system \eqref{eq:main_transcritical} is obtained via the coordinate transformation and parameter rescaling
	\[
	u = - \frac{\tilde u}{\alpha} + v^\ast \left( \frac{\tilde \mu}{2 \mathcal D} \right) - \frac{\kappa}{2 \mathcal D} \tilde \mu, \qquad 
	\mu = \frac{\tilde \mu}{2 \mathcal D}, \qquad
	\eps = \frac{\tilde \eps}{2 \mathcal D} ,
	\]
	where $\kappa = (\beta - \mathcal D) / \alpha$ and $v^\ast(\tilde \mu / \alpha) = \mathcal O(\tilde \mu^2)$ is the unique function satisfying $f(v^\ast(\mu) - \kappa \mu, \mu, 0 ) \equiv 0$ for all $\mu \in \mathcal I$ (which is guaranteed to exist via the Implicit Function Theorem), where $\mathcal I$ is a neighbourhood of $0$ in $\R$.
	
	Assertion 2 can be obtained using
	\[
	u = \frac{\tilde u}{\sqrt{- \gamma}} + v^\ast \left( \frac{\tilde \mu}{\alpha} \right) - \frac{\beta}{\alpha^2} \tilde \mu, \qquad 
	\mu = \frac{\tilde \mu}{\alpha}, \qquad
	\eps = \frac{\tilde \eps}{\alpha} ,
	\]
	where $v^\ast(\tilde \mu / \alpha) = \mathcal O(\tilde \mu^2)$ is the unique function satisfying $f(v^\ast(\mu) - (\beta / \alpha) \mu, \mu, 0 ) \equiv 0$ for all $\mu \in \mathcal I$ (which is guaranteed to exist via the Implicit Function Theorem), where $\mathcal I$ is again a neighbourhood of $0$ in $\R$.
\end{proof}

From this point on we shall work directly with the local normal forms \eqref{eq:main_transcritical} and \eqref{eq:main_pitchfork} in Lemma \ref{lem:normal_forms}. 
%
%
In fact, we will often consider both local normal forms simultaneously. For this purpose, it is convenient to introduce the system
\begin{equation}
	\label{eq:main_system}
	\begin{split}
		\partial_t u &= \partial_x^2 u + \mu u - u^s + \eps \lambda + \mathcal R^{(s)}(u, \mu, \eps) , \\
		\dot \mu &= \eps ,
	\end{split}
\end{equation}
where $s \in \{2, 3\}$, and refer to the case with $s = 2$ ($s = 3$) as the transcritical (pitchfork) case; see Figure \ref{fig:critical_manifolds}. In particular, $\mathcal R^{(2)}$ denotes the higher order terms in system \eqref{eq:main_transcritical}, whereas $\mathcal R^{(3)}$ denotes the higher order terms in system \eqref{eq:main_pitchfork}.

\begin{remark}
	In later sections, we shall reuse some of the notation introduced above for different purposes. In particular, $\alpha, \beta, \gamma$ and $\delta$ will be used to denote new constants (the constants appearing in \eqref{eq:coefficients_transcritical} will not be used again). Note also that the constant $\lambda$ and the higher order terms $\mathcal R$ (we shall often write $\mathcal R$ instead of $\mathcal R^{(s)}$) are different in the transcritical and pitchfork cases; the relevant definition will be clear from the context.
\end{remark}

The limiting problem as $\eps \to 0$ is given by
\begin{equation}
	\label{eq:static_subsystem}
	\partial_t u = \partial_x^2 u + \mu u - u^s + \mathcal R^{(s)} (u,\mu,0) , 
\end{equation}
where
\[
\mathcal R^{(s)} (u,\mu,0) = 
\begin{cases}
	\mathcal O (u^3, u^2 \mu, u \mu^2) , & s = 2, \\
	\mathcal O (u^4, u^2 \mu, u \mu^2) , & s = 3 ,
\end{cases}
\]
and we have chosen to suppress the second equation $\dot \mu = 0$ and view the slow variable $\mu$ as a bifurcation parameter. In both cases, $u = 0$ is a spatially homogenous steady state for all $\mu \in \mathcal I$, where $\mathcal I$ is a neighbourhood about $0$ in $\R$ on which both of the normal form transformations in Lemma \ref{lem:normal_forms} are valid. 
Linearising equation \eqref{eq:static_subsystem} about $u = 0$, we obtain
\begin{equation}
	\label{eq:L}
	\partial_t u = \Lop u , \qquad \Lop := \partial_x^2 + \mu ,
\end{equation}
irrespective of the value of $s$. It is straightforward to show that an instability occurs as the parameter $\mu$ increases over zero.

In order to make this more precise, we need to fix a choice of phase space. To this end, let $k \in \mathbb N$ and denote the space of $k$-differentiable, bounded and uniformly continuous functions on $\R$ by $\mathcal C_{\textup{b,unif}}^k(\R)$, which is a Banach space when endowed with the standard norm
\[
\| v \|_{\mathcal C_{\textup{b,unif}}^k(\R)} = \sum_{i=0}^k \| v^{(i)} \|_\infty ,
\]
where $v^{(0)} = v$, $v^{(i)}$ with $i \geq 1$ denotes the $i$-th derivative of $v$, and $\| \cdot \|_\infty$ is the usual supremum norm (notice in particular that $\| \cdot \|_{\mathcal C^0_{\textup{b,unif}}(\R)} = \| \cdot \|_\infty$). Due to technical difficulties as $|x| \to \infty$, we shall work in spatially polynomially weighted spaces of the form
\begin{equation}
	\label{eq:tilde_Z}
	\widetilde{\mathcal Z} := \left\{ v \in \mathcal C^2_{\textup{b,unif}}(\R) : \| v \|_{\widetilde{\mathcal Z}} < \infty \right\} ,
\end{equation}
where
\[
\| v(\cdot) \|_{\widetilde{\mathcal Z}} = \| v(\cdot) \|_\infty + \| (1 + | \cdot |) v^{(1)} (\cdot) \|_\infty + \| (1 + (\cdot)^2) v^{(2)} (\cdot) \|_\infty .
\]
Thus, we are interested in the behaviour of solutions with initial conditions for which the first (second) derivative decays at least as quickly as $1 / |x|$ ($1 / x^2$) as $|x| \to \infty$. In particular, the limits of functions in $\widetilde{\mathcal Z}$ as $|x| \to \infty$ are constant and bounded, but they need not decay to zero. Viewing $\Lop$ as an operator on $\widetilde{\mathcal Z}$, 
we have that $\Lop \in \mathcal L(\widetilde{\mathcal Z}, \widetilde{\mathcal X})$, where $\mathcal L(X, Y)$ denotes the space of \SJ{bounded/continuous linear operators $\textup{A} : X \to Y$}, and
\begin{equation}
	\label{eq:tilde_X}
	\widetilde{\mathcal X} := \left\{ v \in \mathcal C^0_{\textup{b,unif}}(\R) : \lim_{x \to \pm \infty} v(x) = v_{\pm\infty} \in \R \right\}
\end{equation}
is the space of bounded, uniformly continuous functions with constant limits $v_{\pm \infty}$ as $x \to \pm \infty$ (equipped with the sup norm $\| \cdot \|_\infty$).


\begin{remark}
	\label{rem:density}
	The weighted space $\widetilde{\mathcal Z}$ is a Banach space when endowed with the norm $\| \cdot \|_{\widetilde{\mathcal Z}}$ defined above, $\widetilde{\mathcal X}$ is a Banach space when endowed with the supremum norm, and 
	$\Lop$ is a densely defined operator from $\mathcal D(\Lop) = \widetilde{\mathcal Z} \subseteq \widetilde{\mathcal X}$ to $\widetilde{\mathcal X}$. We refer to Appendix \ref{app:wieghted_spaces} for this and other basic functional-analytic properties of the spaces.
\end{remark}

Direct calculations lead to the following expression for the spectrum:
\[
\sigma(\Lop) := \textup{spec} (\Lop) = (\infty,\mu] .
\]
We omit these calculations here for brevity, but refer to the proof of Lemma \ref{lem:spectrum_K1} below, where we present more details in deriving the spectrum of the (closely related) linear operator $\Lop_1$ defined in \eqref{eq:L1}. It follows that $u^\ast(t) \equiv 0$ is spectrally stable when $\mu < 0$, 
and unstable when $\mu > 0$; see Figure \ref{fig:spectrum_static}. 
Our aim in the following is to understand the fate of solutions of system \eqref{eq:main_system} with $0 < \eps \ll 1$ and initial conditions that are sufficiently small in $\widetilde{\mathcal Z}$ with $\mu(0) < 0$, as they drift through a neighbourhood of the corresponding instability slowly in time.

\begin{figure}[t!]
	\centering
	\subfigure[$\mu < 0$.]{
		\includegraphics[width=0.3\textwidth]{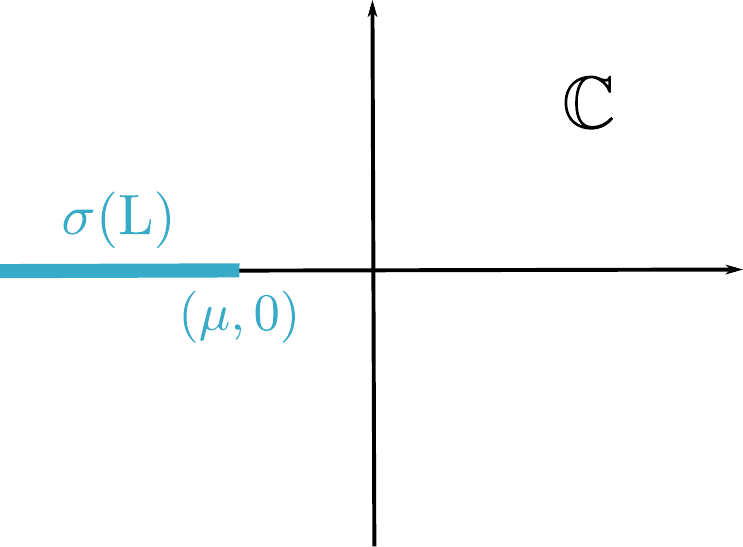}} \quad
	\subfigure[$\mu = 0$.]{
		\includegraphics[width=0.3\textwidth]{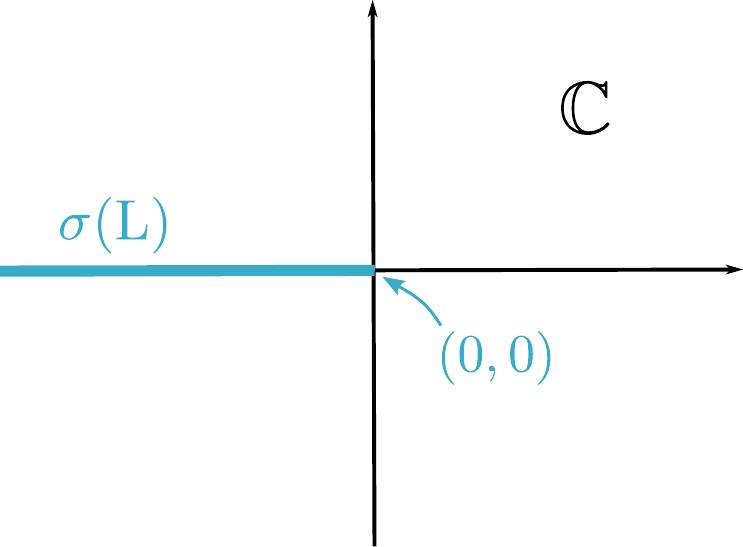}} \quad
	\subfigure[$\mu > 0$.]{
			\includegraphics[width=0.3\textwidth]{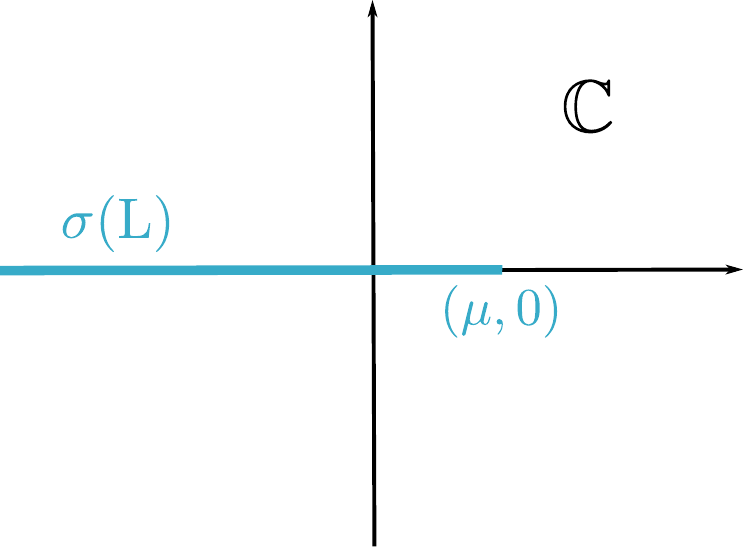}}
	\caption{The spectrum $\sigma(\Lop)$ associated with the linear operator in \eqref{eq:L} for (a) $\mu < 0$, (b) $\mu = 0$ and (c) $\mu > 0$. An instability arises when $\sigma(\Lop)$, which is continuous and constrained to the real axis, intersects the right half plane for $\mu > 0$.}
	\label{fig:spectrum_static}
\end{figure}

\section{Main results}
\label{sec:main_results}

In this section we state and describe our main results. In both transcritical and pitchfork cases, we are interested in the forward evolution of solutions to \eqref{eq:main_system} with initial conditions in
\begin{equation}
	\label{eq:IC_set}
	\Delta^{\textup{in}} := \left\{ (u,-\rho) : \| u(\cdot,0) \|_{\widetilde{\mathcal Z}} \leq \chi \right\}
\end{equation}
where $\rho > 0$ and $\chi > 0$ are small but fixed constants. The primary aim is to describe solutions $u(x,t)$ as they leave an $\eps$-independent neighbourhood of $(u,\mu) = (0,0)$. Specifically, we want a geometric description of the asymptotic behaviour of solutions at time $T > 0$, where $T = 2 \eps^{-1} \rho$ satisfies $\mu(T) = \rho$.

We start by identifying slow manifolds comprised of spatially homogeneous solutions $u(x,t) = \varphi(t)$, which will play an important role in our geometric description of solutions to the PDE systems \eqref{eq:main_transcritical} and \eqref{eq:main_pitchfork}. The existence of slow manifolds follows from the fact that spatially homogenous solutions $(u, \mu) = (\varphi, \mu)(t)$ to system \eqref{eq:main_system} solve the planar slow-fast ODE system
\begin{equation}
	\label{eq:ODEs}
	\begin{split}
		\dot \varphi &= \mu \varphi - \varphi^s + \eps \lambda + \mathcal R^{(s)}(\varphi, \mu, \eps) , \\
		\dot \mu &= \eps ,
	\end{split}
\end{equation}
where, as in system \eqref{eq:main_system}, $s = 2$ ($s = 3$) corresponds to the transcritical (pitchfork) case. The local dynamics of both systems in \eqref{eq:ODEs} have been described using geometric blow-up techniques in \cite{Krupa2001c}. The following result is a direct consequence of this analysis. We let $\mathcal I$ denote the interval of $\mu$-values for which the normal forms in Lemma \ref{lem:normal_forms} are valid (we choose the smallest of the two intervals in order to simplify notation), and let $\mathcal I_- \subset (-\mu_0, 0)$ and $\mathcal I_+ \subset (0,\mu_0)$ be compact subintervals of $\mathcal I$.

\begin{lemma}
	\label{lem:slow_manifolds}
	There exists an $\eps_0 > 0$ and such that the following assertions hold for all $\eps \in (0,\eps_0)$:
	\begin{enumerate}
		\item[(i)] In the transcritical case $s = 2$, system \eqref{eq:ODEs} has two attracting slow manifolds $\mathcal S_a^\pm$ and two repelling slow manifolds $\mathcal S_r^\pm$. In particular,
		\begin{equation}
			\label{eq:slow_manifolds_transcritical}
				\mathcal S_a^- = \left\{ (\mathcal O(\eps), \mu ) : \mu \in \mathcal I_- \right\} , \qquad
				\mathcal S_a^+ = \left\{ (\phi(\mu,\eps), \mu ) : \mu \in \mathcal I_+ \right\} ,
		\end{equation}
		where $\phi$ is a $\mathcal C^k$-smooth function satisfying $\phi(\mu,\eps) = \mu + \mathcal O(\mu^2, \eps)$. 
		\item[(ii)] In the pitchfork case $s = 3$, system \eqref{eq:ODEs} has three attracting slow manifolds $\mathcal S_a$ and $\mathcal S_a^\pm$, and one repelling slow manifold $\mathcal S_r$. In particular,
		\begin{equation}
			\label{eq:slow_manifolds_pitchfork}
				\mathcal S_a = \left\{ (\mathcal O(\eps), \mu ) : \mu \in \mathcal I_- \right\} , \qquad
				\mathcal S_a^\pm = \left\{ (\phi_\pm(\mu,\eps), \mu ) : \mu \in \mathcal I_+ \right\} ,
		\end{equation}
		where the functions $\phi_\pm$ are $\mathcal C^k$-smooth and satisfy $\phi_\pm(\mu,\eps) = \pm \sqrt \mu + \mathcal O(\mu, \eps)$. 
	\end{enumerate}
\end{lemma}

\begin{proof}
	The existence of the slow manifolds in \eqref{eq:slow_manifolds_transcritical} and \eqref{eq:slow_manifolds_pitchfork} follows from Fenichel theory \cite{Fenichel1979,Jones1995,Wiggins2013}; we also refer again to \cite{Krupa2001c}, where both systems have been studied in detail. The asymptotic expressions for $\phi$ ($\phi_\pm$) can be obtained by matching like powers in the power expansions of $\dot \phi(\mu, \eps)$ ($\dot \phi_\pm(\mu, \eps)$) and the right-hand side of the equation for $\dot \varphi$ in \eqref{eq:ODEs} with $\varphi = \phi(\mu, \eps)$ ($\varphi_\pm = \phi_\pm(\mu, \eps)$) about $\eps = 0$. 
\end{proof}


The slow manifolds described in Lemma \ref{lem:slow_manifolds}, which arise in the analysis of the ODE system \eqref{eq:ODEs}, can be identified within the solution space of the PDE system \eqref{eq:main_system}. We have chosen to present it here, because the attracting slow manifolds in particular appear in the formulation of our main results for the PDE systems \eqref{eq:main_transcritical} and \eqref{eq:main_pitchfork}.

\

Our main results describe solutions of system \eqref{eq:main_general} with initial conditions in \eqref{eq:IC_set}. We begin with the transcritical case $s = 2$.

\begin{thm}
	\label{thm:transcritical}
	\textup{(Transcritical exchange of stability)}
	Consider system \eqref{eq:main_transcritical} with $\lambda > 0$. There exists an $\eps_0 > 0$, $\chi > 0$, $\rho > 0$ and $\gamma > 0$ such that for all $\eps \in (0,\eps_0)$, solutions $u(x,t)$ with initial conditions in $\Delta^{\textup{in}}$ satisfy
	\begin{equation}
		\label{eq:asymptotics_transcritical}
		u(x,T) = \phi(\rho,\eps) + \mathcal O(\e^{- \gamma \rho^2 / 2 \eps})
	\end{equation}
	as $\eps \to 0$, where $\phi$ is the function which defines the (spatially homogeneous) slow manifold $\mathcal S_a^+$ in \eqref{eq:slow_manifolds_transcritical}.
\end{thm}

Before discussing this result, we present our main result in the pitchfork case $s = 3$.

\begin{thm}
	\label{thm:pitchfork}
	\textup{(Pitchfork exchange of stability)}
	Consider system \eqref{eq:main_pitchfork} with $\lambda \neq 0$. There exists an $\eps_0 > 0$, $\chi > 0$, $\rho > 0$ and $\gamma > 0$ such that for all $\eps \in (0,\eps_0)$, solutions $u(x,t)$ with initial conditions in $\Delta^{\textup{in}}$ satisfy one of the following assertions:
	\begin{enumerate}
		\item[(i)] If $\lambda > 0$, then
		\begin{equation}
			\label{eq:asymptotics_pitchfork_up}
			u(x,T) = \phi_+(\rho,\eps) + \mathcal O(\e^{- \gamma \rho^2 / 2 \eps})
		\end{equation}
		as $\eps \to 0$, where $\phi_+$ is the function which defines the (spatially homogeneous) slow manifold $\mathcal S_a^+$ in \eqref{eq:slow_manifolds_pitchfork}.
		\item[(ii)] If $\lambda < 0$, then
		\begin{equation}
			\label{eq:asymptotics_pitchfork_down}
			u(x,T) = \phi_-(\rho,\eps) + \mathcal O(\e^{- \gamma \rho^2 / 2 \eps})
		\end{equation}
		as $\eps \to 0$, where $\phi_-$ is the function which defines the (spatially homogeneous) slow manifold $\mathcal S^-_a$ in \eqref{eq:slow_manifolds_pitchfork}.
	\end{enumerate}
\end{thm}

\begin{remark}
	The $\mathcal O$-notation in equations \eqref{eq:asymptotics_transcritical}, \eqref{eq:asymptotics_pitchfork_up} and \eqref{eq:asymptotics_pitchfork_down} is understood in terms of convergence in the $\widetilde{\mathcal Z}$ norm.
\end{remark}

Theorems \ref{thm:transcritical} and \ref{thm:pitchfork} can be viewed as (geometrically formulated) counterparts to the exchange of stability results which apply on bounded spatial domains, as presented in \cite{Butuzov2000}. However, both theorems are proven using a novel adaption of the geometric blow-up method in Section \ref{sec:proofs} below and, from a conceptual and methodological point of view, it is more natural to view them as extensions of the corresponding results on fast-slow ODEs in \cite{Krupa2001c} to the case of scalar reaction-diffusion equations. We reiterate that we consider the method of proof -- as opposed to the particular results in Theorems \ref{thm:transcritical} and \ref{thm:pitchfork} -- to be our main contribution in this article. Nevertheless, Theorems \ref{thm:transcritical} and \ref{thm:pitchfork} are of independent interest, and in the following we collect a number of observations which relate to their interpretation and implications:
\begin{enumerate}[a)]
	\item Equations \eqref{eq:asymptotics_transcritical}, \eqref{eq:asymptotics_pitchfork_up} and \eqref{eq:asymptotics_pitchfork_down} show that solutions are exponentially close to one of the slow manifolds of spatially independent solutions from Lemma \ref{lem:slow_manifolds} in the $\widetilde{\mathcal Z}$ (and therefore also $\mathcal C_{\textup{b,unif}}^2(\R)$) norm when they leave a neighbourhood of the origin. Indeed, the proof in Section \ref{sec:proofs} shows that after an initial transient which occurs for negative $\mu$-values bounded away from zero, the PDE solutions track the extension of $\mathcal S_a^-$ in the case of Theorem \ref{thm:transcritical}, and the extension of $\mathcal S_a$ in the case of Theorem \ref{thm:pitchfork}, through an entire neighbourhood of the origin.
	\item In general, $u(x,T)$ still depends on $x$ via the higher order terms in \eqref{eq:asymptotics_transcritical}, \eqref{eq:asymptotics_pitchfork_up} or \eqref{eq:asymptotics_pitchfork_down}, but this dependence is exponentially small in the $\widetilde{\mathcal Z}$ norm as $\eps \to 0$.
	\item The constant $\gamma > 0$ measures the size of a spectral gap associated with certain linearised problems that arise within the ``exit chart'' of the blow-up. This means that a quantitative estimates for $\gamma$ can be obtained. In the case of Theorem \ref{thm:transcritical}, for any fixed $\gamma \in (0,1)$ there exists an $\eps_0 > 0$ such that the theorem holds. In the case of Theorem \ref{thm:pitchfork}, for any fixed $\gamma \in (0,2)$ there exists an $\eps_0 > 0$ such that the theorem holds. These estimates follow from the proofs of Lemmas \ref{lem:K3_center_manifold_transcritical} and \ref{lem:K3_center_manifold} respectively.
	\item Corresponding statements for the original system \eqref{eq:main_general} can be derived by inverting the linear and near-identity transformations used in the proof of Lemma \ref{lem:normal_forms}. In particular, the `threshold values' of the parameter $\lambda$ are given by \eqref{eq:lambda_transcritical} and \eqref{eq:lambda_pitchfork}.
	\item Theorem \ref{thm:transcritical} only covers the case $\lambda > 0$. The proof in Section \ref{sub:proof_transcritical} shows that solutions are also strongly attracted to a spatially homogeneous slow manifold in the ``entry chart'' when $\mu$ is negative and bounded away from zero when $\lambda < 0$. However, the ODE analysis in \cite{Krupa2001c} shows that the extension of this slow manifold diverges and leads to a `fast escape' scenario; see in particular \cite[Thm.~2.1 Assertion (a)]{Krupa2001c}. This leads to technical issues in the blow-up, which we have not yet been able to resolve (we refer Remark \ref{rem:spherical_blow-up} below for more on this point).
	\item We expect to find spatio-temporal canard phenomena in both systems \eqref{eq:main_transcritical} and \eqref{eq:main_pitchfork}, for $\lambda$-values in an $\eps$-dependent interval about $\lambda = 0$ which shrinks to zero as $\eps \to 0$. This is based on the existence of canard solutions in the ODE system which determines the behaviour of the spatially homogeneous manifolds which `direct' the evolution of the PDE solutions (similarly to the manner in which the slow manifolds identified in Lemma \ref{lem:slow_manifolds} direct PDE solutions with $\lambda > 0$ in Theorem \ref{thm:transcritical} and $\lambda \neq 0$ in Theorem \ref{thm:pitchfork}). We expect an associated delay effect, which could be analysed using an adaptation of arguments developed for problems on bounded spatial domains in \cite{Butuzov2002c,Maesschalck2009}. The question of whether spatio-temporal canard phenomena can be understood via geometric blow-up in this context is open, and left for future work.
\end{enumerate}

\section{Proof via geometric blow-up}
\label{sec:proofs}

In order to resolve the spectral degeneracy described in Section \ref{sec:Classical_bifurcation_in_the_singular_limit}, we consider the extended system obtained by appending the trivial equation $\eps' = 0$ to system \eqref{eq:main_system}:
\begin{equation}
	\label{eq:main_system_extended}
	\begin{split}
		\partial_t u &= \partial_x^2 u + u \mu - u^s + \eps \lambda + \mathcal R^{(s)}(u, \mu, \eps) , \\
		\dot \mu &= \eps , \\
		\dot \eps &= 0, 
	\end{split}
\end{equation}
which we shall hereafter view as an evolution equation on a Banach space. More precisely, we shall work with
\begin{equation}
	\label{eq:spaces}
		\mathcal X := \widetilde{\mathcal X} \times \R^2, \qquad
		\mathcal Y := \widetilde{\mathcal Y} \times \R^2, \qquad
		\mathcal Z := \widetilde{\mathcal Z} \times \R^2,
\end{equation}
where $\widetilde{\mathcal Z}$ is the polynomially weighted space defined in \eqref{eq:tilde_Z}, $\widetilde{\mathcal X}$ is the space defined in \eqref{eq:tilde_X} and $\widetilde{\mathcal Y}$ is a polynomially weighted space defined in a similar fashion to $\widetilde{\mathcal Z}$ although only up to the first derivative, i.e.~
\begin{equation}
	\label{eq:tilde_Y}
	\widetilde{\mathcal Y} := \left\{ v \in \mathcal C^1_{\textup{b,unif}}(\R) : \| v \|_{\widetilde{\mathcal Y}} < \infty \right\} ,
\end{equation}
where
\[
\| v(\cdot) \|_{\widetilde{\mathcal Y}} = \| v(\cdot) \|_\infty + \| (1 + | \cdot |) v^{(1)} (\cdot) \|_\infty .
\]
The spaces $\mathcal X$, $\mathcal Y$ and $\mathcal Z$ are Banach spaces when endowed with the norms
\[
	\| \cdot \|_{\mathcal X} = \| \cdot \|_{\widetilde{\mathcal X}} + | (\cdot, \cdot) | , \qquad
	\| \cdot \|_{\mathcal Y} = \| \cdot \|_{\widetilde{\mathcal Y}} + | (\cdot, \cdot) | , \qquad
	\| \cdot \|_{\mathcal Z} = \| \cdot \|_{\widetilde{\mathcal Z}} + | (\cdot, \cdot) | ,
\]
respectively, and the inclusion map induces a natural continuous embedding $\mathcal Z \hookrightarrow \mathcal Y \hookrightarrow \mathcal X$. We now rewrite system \eqref{eq:main_system_extended} as
\[
\frac{\dd \bm{u}}{\dd t} = \Lop \bm{u} + \Nop (\bm{u}) , 
\]
where $\bm{u} = (u, \mu, \eps)^\transpose \in \mathcal Z$, and $\Lop : \mathcal Z \to \mathcal X$, $\Nop : \mathcal Z \to \mathcal Y$ are linear respectively nonlinear operators defined by
\[
\Lop \bm u = 
\begin{pmatrix}
	\partial_x^2 u + \lambda \eps \\
	\eps \\
	0
\end{pmatrix},
\qquad 
N(\bm u) =
\begin{pmatrix}
	\mu u - u^s + \mathcal R^{(s)}(u,\mu,\eps) \\
	0 \\
	0
\end{pmatrix}.
\]

\begin{remark}
	We choose $\mathcal Y \subset \mathcal C^1_{\textup{b,unif}}(\R)$ instead of $\mathcal Y = \mathcal X$ 
	in \eqref{eq:spaces} because the range of the nonlinearity which appears in the the phase-directional (entry/exit) charts of the blow-up analysis to come includes first derivatives in space. Choosing $\mathcal Y \subset \mathcal C^1_{\textup{b,unif}}(\R)$ allows us to work with the same choice of spaces throughout.
\end{remark}

We now define a blow-up transformation
\[
\Phi : \widetilde{\mathcal Z} \times \mathbb S^1 \times \R_+ \to \mathcal Z 
\]
via
\begin{equation}
	\label{eq:blow-up_map}
	\Phi : \left(\bar u, (\bar \mu, \bar \eps), r \right) \ \mapsto \ 
	\begin{cases}
		u = r \bar u , \\
		\mu = r^{s-1} \bar \mu , \\
		\eps = r^{2 (s-1)} \bar \eps ,
	\end{cases}
\end{equation}
where $\mathbb S^1$ denotes the unit circle; in particular, $\bar \mu^2 + \bar \eps^2 = 1$. We also emphasise that the radial variable $r \geq 0$ depends on time only, i.e.~it is spatially independent, due to the relation $\mu = r^{s-1} \bar \mu$.

\begin{remark}
	The fact that $r$ only depends on time is a consequence of the fact that $\mu = \mu(0) + \eps t$ only depends on time. This is typical in dynamic bifurcation problems, i.e.~in PDEs with `parameters' which slowly vary in time (but not space). In more general PDE settings the radial variable $r$, which measures the `distance' between a solution and the blow-up manifold, will depend on both time and space. This poses a significant analytical obstacle when it comes to the `spatial desingularisation' of the blown-up problem; we refer to Remark \ref{rem:desingularisation} below.
\end{remark}

\begin{remark}
	\label{rem:spherical_blow-up}
	In the ODE counterpart to the problem considered in \cite{Krupa2001c}, a non-normally hyperbolic point at the origin in $\R^3$ (the extended phase space in which $\eps$ is considered as a variable) is blown-up to a $2$-sphere. In contrast, the blow-up transformation defined in \eqref{eq:blow-up_map} blows the set $\{ (u, 0, 0) : u \in \widetilde{\mathcal Z} \}$ up to a `cylinder' 
	which is naturally identified with $\widetilde{\mathcal Z} \times \mathbb S^1$, as shown in Figure \ref{fig:blow-up}. The fact that the blown-up object is an infinite-dimensional object is a direct consequence of the fact that we are working in infinite-dimensional Banach spaces. However, the fact that the blown-up object is `cylindrical' as opposed to `spherical' is a consequence of the fact that we keep $\bar u$ independent from $\bar \mu$ and $\bar \eps$, as opposed to requiring, for example, that
	\begin{equation}
		\label{eq:sphere}
		\bar \mu^2 + \bar \eps^2 + \| \bar u \|^2_{\widetilde{\mathcal Z}} = 1 .
	\end{equation}
	This simplifies the geometry and suffices for the problems considered herein, however it is worthy to note that the incorporation of $\bar u$ into the blow-up via constraints such as \eqref{eq:sphere} is expected to be necessary in the study of other singularities. The question of how to define coordinate charts which in some sense correspond to ``$\bar u = \pm 1$'' in this setting is an open problem.
\end{remark}

\begin{figure*}[t!]
	\centering
	\includegraphics[scale=0.3]{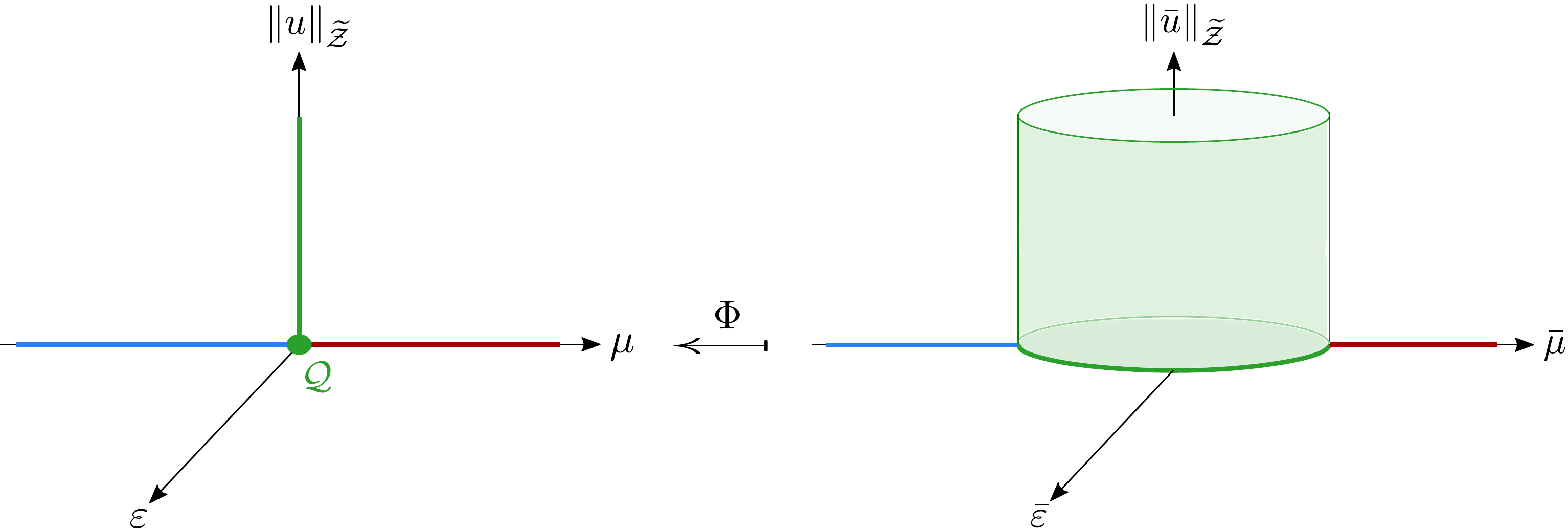}
	\caption{Geometry of the original (blown-down) space (left), and the blown-up space (right) which is obtained via the map $\Phi$ defined in \eqref{eq:blow-up_map}. In order to obtain a geometric representation in $3$ dimensions we show the norms $\| u \|_{\widetilde{\mathcal Z}}$ and $\| \bar u \|_{\widetilde{\mathcal Z}}$ on the vertical axes, as opposed to the variables $u$ and $\bar u$ themselves. The vertical axis in $(\| u \|_{\widetilde{\mathcal Z}}, \mu, \eps)$-space (shown in green) is blown-up to a cylinder in the blown-up $(\| \bar u \|_{\widetilde{\mathcal Z}}, \bar \mu, \bar \eps)$-space (also in green). We also indicate the pre-images of the left and right branches of the $\mu$-axis under $\Phi$ in blue and red respectively (cf.~Figure \ref{fig:critical_manifolds}), as well as the circular pre-image $\Phi^{-1}(\mathcal Q)$ of the singular point $\mathcal Q$ (both are shown in bold green), which corresponds to either $T$ or $P$ in Figure \ref{fig:critical_manifolds}, depending on whether $s = 2$ or $s = 3$ respectively.}
	\label{fig:blow-up}
\end{figure*}

As is usual in blow-up analyses, we shall work in projective coordinates. Three coordinate charts are needed for the analysis, namely,
\[
K_1 : \bar \mu = - 1, \qquad 
K_2 : \bar \eps = 1, \qquad 
K_3 : \bar \mu = 1 ,
\]
which (in accordance with conventions) we will often be referred to as the entry, rescaling and exit charts respectively. We denote local coordinates in each chart by
\[
\begin{split}
	K_1 &: (u, \mu, \eps) = \left( r_1 u_1, -r_1^{s-1}, r_1^{2(s-1)} \eps_1 \right) , \\
	K_2 &: (u, \mu, \eps) = \left( r_2 u_2, r_2^{s-1} \mu_2, r_2^{2(s-1)} \right) , \\
	K_3 &: (u, \mu, \eps) = \left( r_3 u_3, r_3^{s-1}, r_3^{2(s-1)} \eps_3 \right) ,
\end{split}
\]
and note that the change of coordinate formulae are given by
\begin{equation}
	\label{eq:kappa_ij}
	\begin{split}
		\kappa_{12}^{(s)} &: (u_1, r_1, \eps_1) = 
		\left( (-\mu_2)^{- 1/(s-1)} u_2, (-\mu_2)^{1/(s-1)} r_2, \mu_2^{-2} \right) , \\
		\kappa_{23}^{(s)} &: (u_2, r_2, \mu_2) = \left( \eps_3^{- 1 / 2(s-1)} u_3, \eps_3^{1 / 2(s-1)} r_3, \eps_3^{-1/2} \right) ,
	\end{split}
\end{equation}
where the first (second) coordinate change is defined for $\mu_2 < 0$ ($\eps_3 > 0$). 

\subsection{The pitchfork case ($s = 3$)}
\label{sub:proof_pitchfork}

We consider the pitchfork case with $s=3$ first. We proceed as usual by considering the dynamics in each coordinate chart $K_i$, before combining everything in order to prove Theorem \ref{thm:pitchfork}.

\subsubsection{Chart $K_1$}

After rewriting system \eqref{eq:main_system_extended} in chart $K_1$ coordinates and applying a state-dependent transformation of time and space which is defined via the relations
\begin{equation}
	\label{eq:desingularisation_K1}
	\frac{\dd t_1}{\dd t} = r_1^2 , \qquad 
	\frac{x_1}{x} = r_1 ,
\end{equation}
we obtain a set of equations in chart $K_1$ which can be written as
\begin{equation}
	\label{eq:K1_eqns}
	\frac{\dd \bm{u}_1}{\dd t_1} = \Lop_1 \bm{u}_1 + \Nop_1(\bm{u}_1),
\end{equation}
where $\bm{u}_1 = (u_1, r_1, \eps_1)$, $\Lop_1$ is a linear operator defined via 
\begin{equation}
	\label{eq:L1}
	\Lop_1 \bm{u}_1 = 
	\begin{pmatrix}
		\partial_{x_1}^2 u_1 - u_1 \\
		0 \\
		0
	\end{pmatrix} ,
\end{equation}
and $\Nop_1$ is a nonlinear operator defined via
\[
	\Nop_1(\bm{u}_1) = 
	\begin{pmatrix}
		\frac{1}{2} \eps_1 u_1 - u_1^3 + \lambda r_1 \eps_1 - \frac{1}{2} \eps_1 x_1 \partial_{x_1} u_1 + r_1 \mathcal R_1 (u_1, r_1, \eps_1) \\
		- \frac{1}{2} r_1 \eps_1 \\
		2 \eps_1^2
	\end{pmatrix},
\]
where $\mathcal R_1(u_1, r_1, \eps_1) := r_1^{-4} \mathcal R(r_1 u_1, -r_1^2, r_1^4 \eps_1)$, which is well-defined and $\mathcal O(u_1^2, r_1 u_1, r_1^2 \eps_1)$ due to the asymptotics in Lemma \ref{lem:normal_forms} Assertion (ii).

\begin{remark}
	\label{rem:desingularisation}
	The additional transport term $- \tfrac{1}{2} \eps_1 x_1 \partial_{x_1} u_1$ appearing in the right-hand side of system \eqref{eq:K1_eqns} is a consequence of the change of independent variables from $(x,t)$ in system \eqref{eq:main_system_extended}, to $(x_1, t_1)$ in system \eqref{eq:K1_eqns}. Explicitly, the transformation $u(x,t) = r_1(t_1) u_1(x_1, t_1)$, in combination with the relations in \eqref{eq:desingularisation_K1}, implies that
	\begin{equation}
		\label{eq:u1_eqn}
			\partial_t u = \frac{\dd t_1}{\dd t} \partial_{t_1} (r_1 u_1) + \frac{\dd x_1}{\dd t} \partial_{x_1} (r_1 u_1) 
			= r_1^2 \left[ r_1 \partial_{t_1} u_1 + r_1' u_1 + x_1 r_1' \partial_{x_1} u_1 \right] ,
	\end{equation}
	where we use the fact that
	\[
	\frac{\dd x_1}{\dd t} = \frac{\dd x_1}{\dd t_1} \frac{\dd t_1}{\dd t} = r_1 x_1 r_1' .
	\]
	Rearranging and evaluating \eqref{eq:u1_eqn} leads to the equation for $\partial_{t_1} u_1$ in \eqref{eq:K1_eqns}.
	
	\SJ{Similar observations and approaches can be found elsewhere in the literature. In their blow-up analysis of a PDE system posed on a bounded interval $[-a,a]$, \cite{Engel2024,Engel2020} apply a state-dependent rescaling of the form $a = a_1 / r_1$ (or a suitable scaling thereof) instead of a spatial rescaling $x = x_1 / r_1$. This leads to a free boundary problem (instead of introducing an additional transport term).} There is also a formal similarity between the transformation $x = x_1 / r_1$ and time-dependent transformations of space that are commonly applied in order to transform problems with moving boundaries to a coordinate system in which the boundary is fixed; see e.g.~\cite{Tsubota2024} for a recent application of this approach in setting which is closely related to the study of slow passage problems in PDEs.
\end{remark}

\begin{remark}
	\label{rem:desingularisation_2}
	Transformations that are similar to those in \eqref{eq:desingularisation_K1} are common in the study of finite-time blow-up effects by means of \textit{dynamic renormalization}; see e.g.~\cite{Bricmont1995,Chapman2021,Kevrekidis2003,Siettos2003}. In this setting, the left-most transformation in \eqref{eq:desingularisation_K1} would typically be formulated as an integral transformation
	\begin{equation}
		\label{eq:t_transformation}
		t = \int_0^{t_1} \frac{1}{r_1(s)} \dd s .
	\end{equation}
	In our case it suffices to define $t_1$ via in terms of $\dd t / \dd t_1$, which only determines $t_1$ up to an additive constant, since the right-hand side in \eqref{eq:desingularisation_K1} is autonomous and invariant under time translations $t_1 \mapsto t_1 + \alpha$ for any real constant $\alpha$.
	%
\end{remark}

System \eqref{eq:K1_eqns} has a branch of spatially homogeneous steady states
\[
\mathcal P_1 = \left\{ (0, r_1, 0) : r_1 \in [0,\nu] \right\} 
\]
where $\nu > 0$ is a small constant. For $r_1 > 0$, $\mathcal P_1$ corresponds to the branch of steady states along $\{ (0,\mu,0) : \mu < 0 \}$ in the extended system \eqref{eq:main_system_extended} under the blow-up map $\Phi$, viewed in chart $K_1$. Our aim is to apply center manifold theory in a neighbourhood of the point \SJ{$\mathcal Q_1 := (0,0,0) \in \mathcal P_1$}, which lies at the intersection of $\mathcal P_1$ with the blow-up manifold (where $r_1 = 0$).

In order to apply center manifold theory, we need to check three hypotheses, namely, Hypotheses \ref{hyp:1}, \ref{hyp:2} and \ref{hyp:3} in Appendix \ref{app:center_manifold_theory}, which correspond to Hypotheses 2.1, 2.4 and 2.15 in \cite{Haragus2010} respectively. In the following we aim to present a relatively complete treatment of each hypothesis in the particular context of our problem, however we shall refer occasionally to Appendix \ref{app:center_manifold_theory} for details. 
%
%
The first of these -- Hypothesis \ref{hyp:1} -- is easy to verify. Here we simply state the relevant properties without proof:
\begin{enumerate}
	\item[(i)] $\Lop_1 \in \mathcal L(\mathcal Z, \mathcal X)$;
	\item[(ii)] There exists a $k \geq 1$ and a neighbourhood $\mathcal V_1$ of $\mathcal Q_1$ in $\mathcal Z$ such that $\Nop_1 \in \mathcal C^k(\mathcal V_1, \mathcal Y)$, $\Nop_1(\mathcal Q_1) = \bm{0}$ and $\D \Nop_1(\mathcal Q_1) = \bm{0}$.
\end{enumerate}

\begin{remark}
	In order to apply center manifold theory, we also need that $\Lop_1$ is a 
	closed linear operator. 
	This follows from the fact that the resolvant set satisfies $\textup{Res}(\Lop_1) \neq \emptyset$ (due to Lemma \ref{lem:spectrum_K1} below) and the closed graph theorem; see e.g.~\cite[Remark 2.5]{Haragus2010}. 
\end{remark}

\begin{remark}
	The choice to restrict to polynomially weighted spaces $\mathcal Y$ and $\mathcal Z$ for which the first (and in the latter case also the second) derivative decays as $|x_1| \to \infty$ is made so that the requirement $\Nop_1 \in \mathcal C^k(\mathcal V_1, \mathcal Y)$ in property (ii) above is satisfied. In particular, the decay ensures that the transport term $\frac{1}{2} \eps_1 x_1 \partial_{x_1} u_1$ which appears in $\Nop_1(\bm{u}_1)$ is `higher order' over the entire domain.
\end{remark}


%

In order to verify the second hypothesis -- Hypothesis \ref{hyp:2} -- we need to show that the spectrum associated with the linearised problem possesses a spectral gap. This is assured by the following result.

\begin{lemma}
	\label{lem:spectrum_K1}
	The spectrum associated with the linear operator $\Lop_1$ 
	is given by $\sigma(\Lop_1) = \sigma_-(\Lop_1) \cup \sigma_0(\Lop_1)$, where
	\[
	\sigma_-(\Lop_1) = (-\infty, -1], \qquad
	\sigma_0(\Lop_1) = \{ 0 \}.
	\]
	The $0$-eigenvalue has algebraic multiplicity $2$, and the corresponding center subspace $\mathcal E_0$ is $2$-dimensional and spanned by the constant eigenfunctions $(0,1,0)^\transpose$ and $(0,0,1)^\transpose$.
\end{lemma}

\begin{proof}
	We calculate the eigenvalues first. Let $\zeta \in \mathbb C$, $\bm{v_1} = (v_1, v_2, v_3) \in \mathcal Z$ and consider the eigenvalue problem $\Lop_1 \bm{v}_1 = \zeta \bm{v}_1$ or, equivalently,
	\begin{equation}
		\label{eq:eigenvalue_problem_K1}
		\begin{pmatrix}
			v_1''(x_1) - (1 + \zeta) v_1(x_1) \\
			- \zeta v_2 \\
			- \zeta v_3
		\end{pmatrix} 
		= \bm{0} ,
	\end{equation}
	where $v_1''$ denotes the second derivative of $v_1$ with respect to $x_1$. We consider the cases $\zeta = 0$ and $\zeta \neq 0$ separately.
	
	If $\zeta = 0$, then a direct calculation together with the requirement that $v_1 \in \tilde{\mathcal Z}$ is uniformly bounded shows that solutions of \eqref{eq:eigenvalue_problem_K1} must be of the form $\bm{v}_1 = (0, v_2, v_3)$. Since $\bm{v}_1 = (0, v_2, v_3)$ is a solution for any choice of $(v_2, v_3) \in \R^2$, it follows that $0 \in \sigma_0(\Lop_1)$ and that the eigenspace corresponding the $\zeta = 0$ is spanned by $(0,1,0)^\transpose$ and $(0,0,1)^\transpose$.
	
	If $\zeta \neq 0$, then equation \eqref{eq:eigenvalue_problem_K1} immediately yields the requirement $v_2 = v_3 = 0$, implying that eigenfuctions (if they exist) must be of the form $\bm{v}_1 = (v_1, 0, 0)$. Directly solving the second order ODE for $v_1$, we obtain
	\[
	v_1(x_1) =
	\begin{cases}
		c_0, & \zeta = -1 , \\
		c_1 \e^{\sqrt{1 + \zeta} x_1} + c_2 \e^{- \sqrt{1 + \zeta} x_1} , & \zeta \neq -1 ,
	\end{cases}
	\]
	for any constants $c_0, c_1, c_2 \in \R$. It follows that $-1 \in \sigma_-(\Lop_1)$, with constant eigenfunctions of the form $\bm{v}_1 = (c_0, 0, 0)$ for any $c_0 \in \R \setminus \{0\}$. Other values with $\zeta \neq -1$ are ruled out by the additional requirement that $v_1 \in \widetilde{\mathcal Z}$, since the derivatives of $v_1(x_1)$ do not decay as $|x_1| \to \infty$. Thus, the eigenvalues are $\{-1,0\}$.
	
	In order to determine the essential spectrum associated with $\Lop_1$, we need to determine the essential spectrum associated with the linear operator $\widetilde \Lop_1 := \partial_{x_1}^2 - 1 \in \mathcal L(\widetilde{\mathcal Z} , \widetilde{\mathcal X})$. First, we note that the spectrum of the corresponding operator in $\mathcal L(\mathcal C_{\textup{b,unif}}^2(\R), \mathcal C_{\textup{b,unif}}^0(\R))$, which we denote by $\widehat \Lop_1$, is given by $(-\infty, -1]$. This follows from standard results; see e.g.~\cite[Sec.~3]{Sandstede2002} or \cite{Henry1981,Pazy1983}, which guarantee that the point spectrum is empty and that the essential spectrum can be determined by substituting $\bm{u}_1 = \e^{\zeta t_1 + i k x_1} \bm{v_1}$ into $\partial_{t_1} \bm{u}_1 = \widehat \Lop_1 \bm{u}_1$, leading to the dispersion relation
	\[
	\zeta = - 1 - k^2, \qquad k \in \R .
	\]
	Returning to the case at hand, i.e.~to the operator $\widetilde \Lop_1 \in \mathcal L(\widetilde{\mathcal Z} , \widetilde{\mathcal X})$, one can verify directly when $\zeta = 1 - k^2$ for any $k \in \R \setminus \{0\}$, the range $(\Lop_1 - \zeta)(\widetilde{\mathcal Z})$ does not contain the non-zero constant functions. Thus, $(-\infty, -1) \subseteq \sigma(\widetilde L_1)$. For all $\zeta \in \mathbb C \setminus (-\infty, 1]$, i.e.~for all $\zeta$ in the resolvent set for $\widehat \Lop_1$, elements in the range of $(\widehat \Lop_1 - \zeta)^{-1}(\widetilde X)$ are contained in $\mathcal C_{\textup{b,unif}}^2(\R)$ and decay exponentially as $|x_1| \to \infty$; we refer again to \cite[Sec.~3]{Sandstede2002}. Since exponentially decaying functions in $\mathcal C_{\textup{b,unif}}^2(\R)$ are also in $\mathcal{\widetilde Z}$, it follows that $\mathbb C \setminus (-\infty, 1]$ is also contained in the resolvent set of $\widetilde \Lop_1$. Combining this with the calculations for the spectrum of $\Lop_1$ above yields the desired result.
\end{proof}

The spectrum is sketched in Figure \ref{fig:spectral_gap}, and should be compared with Figure \ref{fig:spectrum_static}. Importantly, Lemma \ref{lem:spectrum_K1} shows that the spectral degeneracy identified in the original PDE system \eqref{eq:main_system} -- we refer back to Figure \ref{fig:spectrum_static} -- has been `resolved' in the blown-up space, at least in the entry chart $K_1$, insofar as the blown-up problem features a spectral gap property which is not present in the original blown-down problem. This somewhat remarkable feature of the blow-up transformation is what makes it such a powerful tool.

\begin{remark}
	\label{rem:resolution}
	In the blow-up analysis for the closely related ODE problem in \cite{Krupa2001c}, the loss of (normal) hyperbolicity at the transcritical or pitchfork point is `resolved' in the entry chart $K_1$ in the sense that the vector field is partially hyperbolic at the equilibria. One can also view the `resolution' in this setting as following from the presence of a spectral gap property in blown-up problem.
\end{remark}

\begin{figure}[t!]
	\centering
	\includegraphics[scale=0.5]{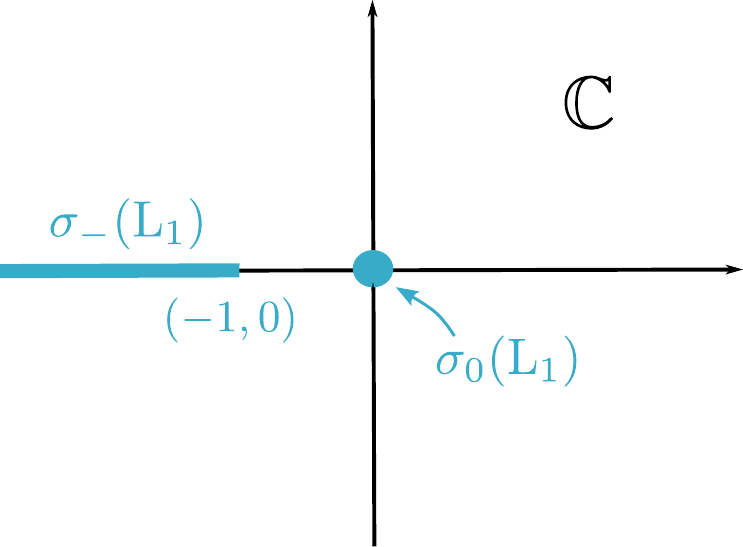}
	\caption{The spectrum $\sigma(\Lop_1)$ associated with the linear operator $\Lop_1$ defined in \eqref{eq:L1}, evaluated at \SJ{$\mathcal Q_1 := (0,0,0)$}. The stable spectrum $\sigma_-(\Lop_1)$ is continuous and bounded below $-1$, and the center spectrum $\sigma_0(\Lop_1) = \{0\}$ has an associated $2$-dimensional center subspace $\mathcal E_0$ which is spanned by constant eigenfunctions. In contrast to the corresponding spectrum for the blown-down problem shown in Figure \ref{fig:spectrum_static}(b), there is a spectral gap.}
	\label{fig:spectral_gap}
\end{figure}

We turn now to the third and final hypothesis -- Hypothesis \ref{hyp:3}. In our case, this is equivalent to satisfying the pair of resolvent bounds described in the following result.

\begin{lemma}
	\label{lem:resolvent_K1}
	For any $\omega \in \R$ such that $| \omega | \geq 1$, $i \omega \in \textup{Res} (\Lop_1)$ and there exists a constant $c > 0$ such that the following resolvent bounds are satisfied:
	\begin{equation}
		\label{eq:resolvent_bounds_K1}
			\| (i \omega \textup{Id} - \Lop_1)^{-1} \|_{\mathcal L(\mathcal X)} \leq \frac{c}{|\omega|} , \qquad
			\| (i \omega \textup{Id} - \Lop_1)^{-1} \|_{\mathcal L(\mathcal Y, \mathcal Z)} \leq \frac{c}{|\omega|^{1/2}} .
	\end{equation}
\end{lemma}

\begin{proof}
	The fact that $i \omega \in \textup{Res} (\Lop_1)$ for all $\omega \in \R$ such that $| \omega | \geq 1$ follows immediately from the fact that $\sigma_0(\Lop_1) = \{0\}$, as shown in Lemma \ref{lem:spectrum_K1} (in fact $i \omega \in \textup{Res} (\Lop_1)$ for all $\omega \in \R \setminus \{ 0 \}$). It therefore remains to verify the bounds in \eqref{eq:resolvent_bounds_K1}.
	
	Since $i \omega \in \textup{Res}(\Lop_1)$, the operator $i \omega \textup{Id} - \Lop_1$ is invertible. Thus for any $\bm{b} = (b_1, b_2, b_3) \in \mathcal X$ (known) and $\bm a = (a_1, a_2, a_3) \in \mathcal Z$ (unknown), the equation $(i \omega \textup{Id} - \Lop_1) \bm{a} = \bm{b}$ has a unique solution. A direct computation using the definition of $\Lop_1$ in \eqref{eq:L1} shows that $a_2 = - i b_2 / \omega$ and $a_3 = - i b_3 / \omega$, whereas $a_1$ is (uniquely) determined by the equation
	\begin{equation}
		\label{eq:a1_ode}
		a_1''(x_1) - (1 + i \omega) a_1(x_1) = - b_1(x_1) ,
	\end{equation}
	together with the requirement that $a_1 \in \widetilde{\mathcal Z}$. A somewhat lengthy but nevertheless direct computation shows that the unique bounded solution is given by
	\begin{equation}
		\label{eq:a1}
		a_1(x_1) = \frac{\e^{\eta x_1}}{2 \eta} \int_{x_1}^\infty \e^{- \eta \xi} b_1(\xi) d \xi + \frac{\e^{- \eta x_1}}{2 \eta} \int_{-\infty}^{x_1} \e^{\eta \xi} b_1(\xi) d \xi ,
	\end{equation}
	where $\eta := \sqrt{1 + i \omega}$; we refer to the Masters thesis \cite{Martinez-Sanchez2024} for details. In order to prove the first inequality in \eqref{eq:resolvent_bounds_K1}, we show that
	\begin{equation}
		\label{eq:ab_bound}
		\frac{\| \bm{a} \|_{\mathcal X}}{\| \bm{b} \|_{\mathcal X}} 
		\leq \frac{\| a_1 \|_\infty}{\| b_1 \|_\infty} + \frac{|(a_2,a_3)|}{|(b_2, b_3)|} 
		\leq \frac{c}{| \omega |}
	\end{equation}
	for some constant $c > 0$ (the left-most inequality follows directly from the definition of $\| \cdot \|_{\mathcal X}$, whereas the right-most needs to be shown). Using the expressions for $a_2$ and $a_3$ obtained above, we have that
	\begin{equation}
		\label{eq:ab_bound_Euclidean_part}
		\frac{|(a_2,a_3)|}{|(b_2, b_3)|} = 
		\frac{|(- i b_2 / \omega, - i b_3 / \omega)|}{|(b_2, b_3)|} =
		\frac{1}{| \omega |} ,
	\end{equation}
	so it remains to bound the left-most term of the middle expression in \eqref{eq:ab_bound}. One can show directly with a triangle inequality that the integrals appearing in the formula for $a_1(x_1)$ in \eqref{eq:a1} can be both be bounded by $\| b_1 \|_\infty / (2 | \eta |^2)$, so that
	\[
	\frac{\| a_1 \|_\infty}{\| b_1 \|_\infty} + \frac{|(a_2,a_3)|}{|(b_2, b_3)|} \leq \frac{1}{2 |\eta|^2} + \frac{1}{2 |\eta|^2} \leq \frac{1}{|\omega|} .
	\]
	Combining this with the bound in \eqref{eq:ab_bound_Euclidean_part} shows that \eqref{eq:ab_bound} is satisfied with $c = 2$.
	
	In order to prove the second inequality in \eqref{eq:resolvent_bounds_K1}, we need to show that
	\begin{equation}
		\label{eq:ab_bound_2}
		\frac{\| \bm{a} \|_{\mathcal Z}}{\| \bm{b} \|_{\mathcal Y}} 
		\leq \frac{\| a_1 \|_{\widetilde{\mathcal Z}}}{\| b_1 \|_{\widetilde{\mathcal Y}}} + \frac{|(a_2,a_3)|}{|(b_2, b_3)|} 
		\leq \frac{c}{| \omega |^{1/2}}
	\end{equation}
	for a constant $c > 0$ (which may of course be different from the constant $c$ appearing in the first bound in \eqref{eq:resolvent_bounds_K1}). Since the second term in the middle expression can again be bounded by \eqref{eq:ab_bound_Euclidean_part}, and $1 / | \omega | \leq 1 / | \omega |^{1/2}$ since $| \omega | \geq 1$, it remains to show that
	\[
		\frac{\| a_1 \|_{ \widetilde{\mathcal Z}}}{\| b_1 \|_{\widetilde{\mathcal Y}}} =
		\frac{\| a_1 \|_\infty + \| (1 + |\cdot|) a_1' \|_\infty + \| (1 + (\cdot)^2) a_1'' \|_\infty}{\| b_1 \|_\infty + \| (1 + |\cdot|) b_1' \|_\infty} 
		\leq \frac{\tilde c}{|\omega|^{1/2}} 
	\]
	for some constant $\tilde c > 0$. We shall actually show that there exists a $\tilde c > 0$ such that the second inequality in
	\begin{equation}
		\label{eq:ab_bound_3}
			\frac{\| a_1 \|_{ \widetilde{\mathcal Z}} }{\| b_1 \|_{ \widetilde{\mathcal Y}} } 
			\leq 
			\frac{\| a_1 \|_\infty}{\| b_1 \|_\infty} + 
			\frac{\| (1 + | \cdot |) a_1' \|_\infty}{\| (1 + |\cdot|) b_1' \|_\infty} + \frac{\| (1 + (\cdot)^2 ) a_1'' \|_\infty}{\| (1 + | \cdot |) b_1' \|_\infty} \\
			\leq \frac{\tilde c}{|\omega|^{1/2}} 
	\end{equation}
	is satisfied (the first follows directly from the preceding expression). We have already shown that $\| a_1 \|_\infty / \| b_1 \|_\infty \leq 1 / | \omega | \leq 1 / | \omega|^{1/2}$, 
	so it remains to bound the latter two terms in the middle expression.
	
	The derivatives $a_1'$ and $a_1''$ are given by
	\[
	\begin{split}
		a_1'(x_1) &= \frac{\e^{\eta x_1}}{2 \eta} \int_{x_1}^\infty \e^{-\eta \xi} b_1'(\xi) d \xi + \frac{\e^{- \eta x_1}}{2 \eta} \int_{-\infty}^{x_1} \e^{\eta \xi} b_1'(\xi) d \xi , \\
		a_1''(x_1) &= \frac{\e^{\eta x_1}}{2} \int_{x_1}^\infty \e^{-\eta \xi} b_1'(\xi) d \xi - \frac{\e^{- \eta x_1}}{2} \int_{-\infty}^{x_1} \e^{\eta \xi} b_1'(\xi) d \xi , 
	\end{split}
	\]
	where the former is obtained by differentiating \eqref{eq:a1} and integrating by parts, and the latter can be determined by rearranging \eqref{eq:a1_ode} for $a_1''(x_1)$, inserting \eqref{eq:a1} and integrating by parts; we refer to \cite[pp.43-45]{Haragus2010} for similar calculations in the context of a parabolic reaction-diffusion equation on a bounded domain. 
	Using the expression for $a_1'(x_1)$, we obtain the bound
	\[
	\begin{split}
		\| (1 + |\cdot| ) a_1' \|_\infty &= \sup_{x_1 \in \R} \| (1 + |x_1|) a_1'(x_1) \| \\
		&\leq \frac{1}{2 |\eta|} \sup_{x_1 \in \R} \left| (1 + |x_1|) \e^{\eta x_1} \int_{x_1}^\infty \e^{- \eta \xi} b_1'(\xi) d \xi \right| + \frac{1}{2 |\eta|} \sup_{x_1 \in \R} \left| (1 + |x_1|) \e^{- \eta x_1} \int_{-\infty}^{x_1} \e^{\eta \xi} b_1'(\xi) d \xi \right| \\
		&\leq \frac{\| (1 + | \cdot |) b_1' \|_\infty}{2 |\eta|} 
		\left( \e^{\eta_r x_1} \int_{x_1}^\infty \e^{- \eta_r \xi} d \xi + \e^{- \eta_r x_1} \int_{- \infty}^{x_1} \e^{ \eta_r \xi} d \xi \right) \\
		&= \frac{\| (1 + | \cdot |) b_1' \|_\infty}{| \eta | \eta_r }, 
	\end{split}
	\]
	where $\eta_r = \text{Re}(\eta)$. Using $\eta = \sqrt{1 + i \omega}$ one can show that $\eta_r = |\omega|^{1/2} / \sqrt{2}$, which yields
	\[
	\frac{\| (1 + | \cdot |) a_1' \|_\infty}{\| (1 + |\cdot|) b_1' \|_\infty} \leq \frac{\sqrt{2}}{| \omega |^{1/2}} .
	\]
	
	It therefore remains to bound the right-most term in the first line of \eqref{eq:ab_bound_3}. This can be done using the expression for $a''(x_1)$ derived above, however, we have opted to defer the details to Lemma \ref{lem:resolvant_bound}, Appendix \ref{app:bounds}, in order to preserve clarity in the overall argument within the main text. Combining the bound in Lemma \ref{lem:resolvant_bound} with the bounds obtained above proves \eqref{eq:ab_bound_3}, thereby completing the proof.
\end{proof}

Having verified Hypotheses \ref{hyp:1}-\ref{hyp:3}, we are now in a position to apply the center manifold theory developed in \cite{Haragus2010}. This leads to the following result.

\begin{lemma}
	\label{lem:K1_center_manifold}
	There exists a neighbourhood $\mathcal U_1$ of $\mathcal Q_1$ in $\mathcal Z$ such that system \eqref{eq:K1_eqns} has a 2-dimensional local center manifold
	\begin{equation}
		\label{eq:center_manifold_K1}
		\mathcal M_1^a = \left\{ (u_1, r_1, \eps_1) \in \mathcal U_1 : u_1 = \Psi_1(r_1, \eps_1) \right\} ,
	\end{equation}
	which is locally invariant and tangent to $\mathcal E_0$ at $\mathcal Q_1$, for which the reduction function $\Psi_1 \in \mathcal C^k(\mathcal E_0, \widetilde{\mathcal Z})$ has asymptotics
	\begin{equation}
		\label{eq:Psi_1}
		\Psi_1(r_1, \eps_1) = \lambda r_1 \eps_1 + \mathcal O( |(r_1, \eps_1) |^3 ) .
	\end{equation}
	Moreover, $\mathcal M_1^a$ is locally attracting: if $\bm{u}_1(0) \in \mathcal U_1$ and $\bm u_1(t_1; \bm{u}_1(0)) \in \mathcal U_1$ for all $t_1 > 0$, then there exists a `base function' ${\bm{\tilde u}_1} \in \mathcal M_1^a \cap \mathcal U_1$ and a $\gamma > 0$ such that
	\[
	\bm{u}_1(t_1; \bm{u}_1(0)) = \bm{u}_1 (t_1; \tilde{\bm{u}}_1) + \mathcal O(\e^{- \gamma t_1}) 
	\]
	as $t_1 \to \infty$.
\end{lemma}

\begin{proof}
	The existence of a 2-dimensional local center manifold $\mathcal M_1^a$ of the form \eqref{eq:center_manifold_K1} which is tangent to $\mathcal E_0$ at $\mathcal Q_1$ follows from \cite[Thm.~2.9]{Haragus2010}, and the attractivity property follows from \cite[Thm.~3.22]{Haragus2010}.
	
	It remains to show that the reduction function $\Psi_1$ has the asypmtotics in \eqref{eq:Psi_1}. First, we note that $\Psi_1(r_1,\eps_1)$ is spatially independent; this is a direct consequence of the fact that the variables $r_1$ and $\eps_1$ are spatially independent. It follows that $\Psi_1$ satisfies the ODE system
	\begin{equation}
		\label{eq:K1_ODEs}
		\begin{split}
			\Psi_1' &= - \left( 1 -\frac{\eps_1}{2} \right) \Psi_1 - \Psi_1^3 + r_1 \left( \lambda \eps_1 + \mathcal R_1 (\Psi_1, r_1, \eps_1) \right) , \\
			r_1' &= - \frac{1}{2} r_1 \eps_1 , \\
			\eps_1' &= 2 \eps_1^2 ,
		\end{split}
	\end{equation}
	where the prime denotes differentiation with respect to $t_1$. Substituting a power series ansatz for $\Psi_1(r_1,\eps_1)$ and equating like powers in the corresponding invariance equation leads to the asymptotics in \eqref{eq:Psi_1}.
\end{proof}

The $2$-dimensional local center manifold $\mathcal M_1^a$ is sketched in the global geometry of the blown-up space in Figure \ref{fig:center_manifolds}. By Lemma \ref{lem:K1_center_manifold}, solutions to the PDE system \eqref{eq:K1_eqns} which have initial conditions in $\mathcal V_1$ are strongly attracted to the local center manifold $\mathcal M_1^a$, after which they track solutions which are determined by the $2$-dimensional ODE system
\[
\begin{split}
	r_1' &= - \frac{1}{2} r_1 \eps_1 , \\
	\eps_1' &= 2 \eps_1^2 .
\end{split}
\]
This is a significant simplification, especially because this system can be solved explicitly by direct integration. Here it is worthy to note that the ODE system \eqref{eq:K1_ODEs} is precisely the ODE system which appears in within the entry chart analysis of the corresponding ODE problem in \cite{Krupa2001c} (up to higher order corrections which follow from the presence of additional terms in the slow equation in \cite{Krupa2001c}). The primary conceptual distinction between the center manifold appearing in the analysis in \cite{Krupa2001c} and the center manifold $\mathcal M_1^a$, is that $\mathcal M_1^a$ is embedded and attracting in the (infinite-dimensional) Banach space $\widetilde{\mathcal Z}$.

\begin{figure}[t!]
	\centering
	\includegraphics[scale=0.35]{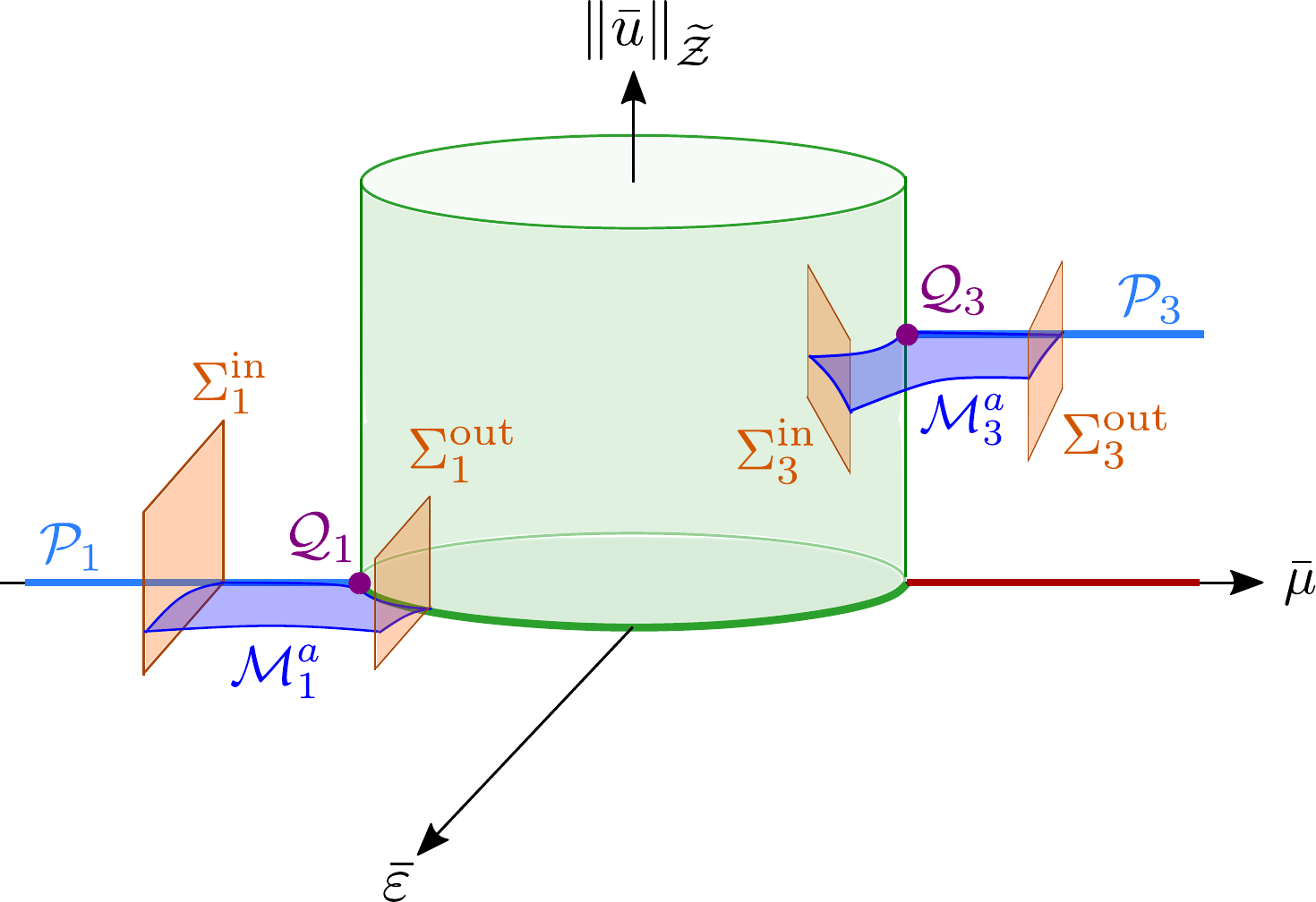}
	\caption{Sketch of the important objects appearing in the local analysis in charts $K_1$ and $K_3$. In $K_1$ we identify the equilibrium $\mathcal Q_1$ (shown in purple), which lies at the intersection of the blow-up cylinder with the branch of constant steady states $\mathcal P_1$ (shown in light blue). Lemma \ref{lem:K1_center_manifold} asserts the existence of a $2$-dimensional local center manifold $\mathcal M_1^a$ near $\mathcal Q_1$, shown here in blue, which coincides with $\mathcal P_1$ when restricted to $\eps_1 = 0$. The entry and exit sections $\Sigma_1^{\textup{in}}$ and $\Sigma_3^{\textup{out}}$ which are used in order to define the local transition map $\pi_1$ are shown in orange. We also sketch the corresponding objects in chart $K_3$. Note that we simply write $\mathcal Q_3$, $\mathcal P_3$, $\mathcal M_3^a$, $\Sigma_3^{\textup{in}}$ and $\Sigma_3^{\textup{out}}$ to indicate either $\mathcal Q_3^+$, $\mathcal P_3^+$, $\mathcal M_3^{a,+}$, $\Sigma_3^{\textup{in},+}$ and $\Sigma_3^{\textup{out},+}$ or $\mathcal Q_3^-$, $\mathcal P_3^-$, $\mathcal M_3^{a,-}$, $\Sigma_3^{\textup{in},-}$ and $\Sigma_3^{\textup{out},-}$, i.e.~we do not distinguish between `$\pm$ cases'. This is because we plot the norm on the vertical axis (as opposed to the variable $\bar u$ itself), and therefore expect a similar picture in each case.}
	\label{fig:center_manifolds}
\end{figure}

It remains to use Lemma \ref{lem:K1_center_manifold} in order to describe the key dynamics in chart $K_1$. Our aim is to describe the transition map $\pi_1 : \Sigma_1^{\textup{in}} \to \Sigma_1^{\textup{out}}$ induced by the forward evolution of initial conditions in
\[
\Sigma_1^{\textup{in}} := \left\{(u_1, \nu, \eps_1) : \| u_1 \|_{\widetilde{\mathcal Z}} \leq \beta, \eps_1 \in [0, \delta] \right\}
\]
up to
\[
\Sigma_1^{\textup{out}} := \left\{(u_1, r_1, \delta) : \| u_1 \|_{\widetilde{\mathcal Z}} \leq \beta, r_1 \in [0, \nu] \right\} ,
\]
see Figure \ref{fig:center_manifolds}. Note that $\Sigma_1^{\textup{in}}$ can be chosen to contain the set $\Delta^{\textup{in}} \times [0,\eps_0]$ after `blowing down' if we set $\nu = \rho^{1/2}$, $\delta = \eps_0 / \rho^2$ and decrease the size of $\chi$ if necessary ($\| u \|_{\widetilde{\mathcal Z}}$ and $\| u_1 \|_{\widetilde{\mathcal Z}}$ can be compared using $x_1 = \nu x$ and the scaling properties of the norm described in Lemma \ref{lem:norm_equivalence}, Appendix \ref{app:wieghted_spaces}). 

\begin{proposition}
	\label{prop:K1_dynamics}
	For fixed but sufficiently small $\beta, \delta, \rho > 0$, $\pi_1$ is well-defined and given by
	\[
	\pi_1 : 
	\begin{pmatrix}
		u_1 \\
		\nu \\
		\eps_1
	\end{pmatrix} \mapsto 
	\begin{pmatrix}
		\Psi_1 \left(\nu (\eps_1 / \delta)^{1/4}, \delta \right) + \mathcal O(\e^{-\gamma / 2 \eps_1}) \\
		\nu (\eps_1 / \delta)^{1/4} \\
		\delta
	\end{pmatrix} ,
	\]
	where $\Psi_1$ is the function which defines the center manifold $\mathcal M_1^a$, and $\gamma > 0$ is the constant which appears in Lemma \ref{lem:K1_center_manifold}.
\end{proposition}

\begin{proof}
	The ODEs for $(r_1, \eps_1)$ decouple, and can be solved by direct integration. We also need an estimate for the transition time $T_1$, which can be determined using the constraint $\eps_1(t_1 = T_1) = \delta$. We obtain
	\[
	T_1 = \frac{1}{2 \eps_1(0)} - \frac{1}{2 \delta} .
	\]
	Using the attractivity of the center manifold $\mathcal M_1^a$, we obtain
	\[
	u_1(T_1) = \Psi_1 \left(r_1(T_1), \eps_1(T_1) \right) + \mathcal O \left( \e^{- \gamma T_1} \right) ,
	\]
	which yields the desired result after substituting the values for $T_1$, $r_1(T_1)$ and $\eps_1(T_1)$.
\end{proof}

We note that the ``$\mathcal O$'' notation in Proposition \ref{prop:K1_dynamics} and its proof should be understood with respect to the $\widetilde{\mathcal Z}$ norm. In particular, Proposition \ref{prop:K1_dynamics} shows that solutions of system \eqref{eq:K1_eqns} are exponentially close to the spatially independent solutions in $\mathcal M_1^a \cap \Sigma_1^{\textup{out}}$ when they leave a neighbourhood of $\mathcal Q_1$ via $\Sigma_1^{\textup{out}}$.

\subsubsection{Dynamics in the exit chart $K_3$}

We now consider the dynamics in the exit chart $K_3$. Calculations which are similar to those which led to the system \eqref{eq:K1_eqns} in chart $K_1$, involving a state-dependent transformation of time and space defined via the relations
\[
\frac{\dd t_3}{\dd t} = r_3^2, \qquad \frac{x_3}{x} = r_3 ,
\]
lead to the following system in chart $K_3$:
\begin{equation}
	\label{eq:K3_eqns}
	\frac{\dd \bm{u}_3}{\dd t_3} = \Lop_3 \bm{u}_3 + \Nop_3(\bm{u}_3),
\end{equation}
where $\bm{u}_3 = (u_3, r_3, \eps_3)^\transpose \in \mathcal Z$, and $\Lop_3$, 
$\Nop_3$ 
are linear respectively nonlinear operators defined by
\[
\Lop_3 \bm{u}_3 = 
\begin{pmatrix}
	\partial_{x_3}^2 u_3 + u_3 \\
	0 \\
	0
\end{pmatrix}
\]
and
\[
	N_3(\bm{u}_3) = 
	\begin{pmatrix}
		- \frac{1}{2} \eps_3 u_3 - u_3^3 + \lambda r_3 \eps_3 - \frac{1}{2} \eps_3 x_3 \partial_{x_3} u_3 + r_3 \mathcal R_3(u_3, r_3, \eps_3) \\
		\frac{1}{2} r_3 \eps_3 \\
		- 2 \eps_3^2
	\end{pmatrix},
\]
where $\mathcal R_3 (u_3, r_3, \eps_3) = r_3^{-4} \mathcal R(r_3 u_3, r_3^2, r_3^4 \eps_3)$, which is well-defined and $\mathcal O(u_3^2, r_3 u_3, r_3^2 \eps_3)$ due to the asymptotics in Lemma \ref{lem:normal_forms}, Assertion (ii).

Here we identify two distinct branches of spatially homogeneous steady states in system \eqref{eq:K3_eqns}, namely
\begin{equation}
	\label{eq:P3}
	\mathcal P_3^\pm = \left\{ (\pm 1 + \mathcal O(r_3), r_3, 0) : r_3 \in [0, \nu] \right\} 
\end{equation}
which, for $r_3 > 0$, correspond to the branches of the critical manifolds along $\{ (\phi_\pm(\mu,0),\mu,0) : \mu > 0 \}$ (see \eqref{eq:slow_manifolds_pitchfork}) in the extended system \eqref{eq:main_system_extended} under the blow-up map $\Phi$, viewed in chart $K_3$.

\begin{remark}
	\label{rem:K3_steady_states}
	System \eqref{eq:K3_eqns} also has additional steady states, including a family of (unstable) spatially homogenous steady states for $u_3 = 0$, $\eps_3 = 0$ and $r_3 \geq 0$, a family of spatially oscillatory steady states, and two distinct standing waves (all within $r_1 = \eps_1 = 0$). We do not consider these further, however, since they will not play a role in the analysis.
\end{remark}

Similarly to the analysis in chart $K_1$, we can apply center manifold theory in a neighbourhood of the points $\mathcal Q_3^\pm := (\pm 1,0,0) \in \mathcal P_3^\pm$, which lie within the intersection of $\mathcal P_3^\pm$ with the blow-up manifold (where $r_3 = 0$).

\begin{lemma}
	\label{lem:K3_center_manifold}
	There are neighbourhoods $\mathcal U_3^\pm$ of $\mathcal Q_3^\pm$ in $\mathcal Z$ such that system \eqref{eq:K3_eqns} has 2-dimensional local center manifolds
	\begin{equation}
		\label{eq:center_manifold_K3}
		\mathcal M_3^{a,\pm} = \left\{ (u_3, r_3, \eps_3) \in \mathcal U_3 : u_3 = \Psi_3^\pm(r_3, \eps_3) \right\} ,
	\end{equation}
	which are locally invariant and tangent to $\mathcal E_0(\mathcal Q_3^\pm)$ at $\mathcal Q_3^\pm$, for which the reduction functions $\Psi_3^\pm \in \mathcal C^k(\mathcal E_0(\mathcal Q_3^\pm), \widetilde{\mathcal Z})$ have the asymptotics
	\begin{equation}
		\label{eq:Psi_3}
		\Psi_3^\pm(r_3, \eps_3) = \pm 1 \mp \frac{1}{4} \eps_3 + \frac{1}{2} \mathcal R_3 (\pm 1,0,0) r_3 
		+ \mathcal O( \|(r_3, \eps_3) \|^2 ) .
	\end{equation}
	Moreover, both $\mathcal M_3^{a,\pm}$ are locally attracting: if $\bm{u}_3(0) \in \mathcal U_3^\pm$ and $\bm u_3(t_3; \bm{u}_3(0)) \in \mathcal U_3^\pm$ for all $t_3 > 0$, then there exists a `base function' ${\bm{\tilde u_3}} \in \mathcal M_3^{a,\pm} \cap \mathcal U_1^\pm$ and a $\gamma > 0$ such that
	\[
	\bm{u}_3(t_3; \bm{u}_3(0)) = \bm{u}_3 (t_3; \tilde{\bm{u}}_3) + \mathcal O(\e^{- \gamma t_3}) 
	\]
	as $t_3 \to \infty$.
\end{lemma}

\begin{proof}
	We prove the statement pertaining to the center manifold $\mathcal M_3^{a,+}$ in a neighbourhood of $\mathcal Q_3^+$; the proof of the statement pertaining to $\mathcal M_3^{a,-}$ is analogous.
	
	We start by applying the coordinate transformation
	\begin{equation}
		\label{eq:K3_coord_change}
		\tilde u_3 = u_3 - 1 + \frac{1}{4} \eps_3 - \frac{1}{2} \mathcal R_3 (1,0,0) r_3 ,
	\end{equation}
	which (i) translates the point $\mathcal Q_3^+$ to the origin, and (ii) diagonalises the linear part. Explicitly, this leads to
	\[
	\frac{\dd \tilde{\bm{u}}_3}{\dd t_3} = \widetilde{L}_3 \tilde{\bm{u}}_3 + \widetilde \Nop_3(\tilde{\bm{u}}_3)
	\]
	where $\tilde{\bm{u}}_3 = (\tilde u_3, r_3, \eps_3)$,
	\[
	\widetilde \Lop_3 \tilde{\bm{u}}_3 = 
	\begin{pmatrix}
		\partial_{x_3}^2 \tilde u_3 - 2 \tilde u_3 \\
		0 \\
		0
	\end{pmatrix} ,
	\]
	and
	\[
		N_3(\tilde{\bm{u}}_3) = 
		\begin{pmatrix}
			- 3 \tilde u_3^2 - \tilde u_3^3 - \frac{1}{2} \eps_3 x_3 \partial_{x_3} u_3 + \mathcal O(\tilde u_3 \eps_3, r_3 \tilde u_3, r_3 \eps_3^2, r_3 \eps_3, r_3^2 ) \\
			\frac{1}{2} r_3 \eps_3 \\
			- 2 \eps_3^2
		\end{pmatrix}.
	\]
	
	In order to apply center manifold theory, we proceed by verifying Hypotheses \ref{hyp:1}-\ref{hyp:3}, as in the $K_1$ analysis above. Straightforward calculations show that
	\begin{enumerate}
		\item[(i)] $\widetilde \Lop_3 \in \mathcal L(\mathcal Z, \mathcal X)$;
		\item[(ii)] There exists a $k \geq 1$ and a neighbourhood $\mathcal V_3^+$ of \SJ{$\widetilde{\mathcal Q}_3^+ := (0,0,0)$} in $\mathcal Z$ such that $\widetilde \Nop_3 \in \mathcal C^k(\mathcal V_3, \mathcal Y)$, $\widetilde \Nop_3(\widetilde{\mathcal Q}_3^+) = \bm{0}$ and $\D \widetilde \Nop_3(\widetilde{\mathcal Q}_3^+) = \bm{0}$,
	\end{enumerate}
	showing that Hypothesis \ref{hyp:1} is satisfied.
	
	In order to verify Hypothesis \ref{hyp:2}, we need to calculate the spectrum $\sigma(\widetilde L_3)$. Direct calculations that are analogous to those appearing in the proof of Lemma \ref{lem:spectrum_K1} show that $\sigma(\widetilde L_3) = \sigma_-(\widetilde L_3) \cup \sigma_0(\widetilde L_3)$, where
	\[
	\sigma_-(\widetilde L_3) = (-\infty, -2], \qquad 
	\sigma_0(\widetilde L_3) = \{0\} .
	\]
	As in Lemma \ref{lem:spectrum_K1}, the zero eigenvalue has algebraic multiplicity $2$ and the corresponding center subspace $\mathcal E_0$ is $2$-dimensional and spanned by the constant eigenfunctions $(0,1,0)^\transpose$ and $(0,0,1)^\transpose$. Thus, Hypothesis \ref{hyp:2} is satisfied.
	
	Hypothesis \ref{hyp:3} can be verified by proving a result analogous to Lemma \ref{lem:resolvent_K1} for the resolvent associated with $\widetilde L_3$. This too can be done via analogous arguments (due to the very similar structure of $\Lop_1$ and $\widetilde L_3$), so we omit the details, referring to the proof of Lemma \ref{lem:resolvent_K1} for details.
	
	Thus, all three hypotheses are satisfied, allowing us to infer the existence of a $2$-dimensional local center manifold which can be written as a graph of the form $\tilde u_3 = \widetilde \Psi_3^+(r_3,\eps_3)$ (existence and attractivity follow from \cite[Thm.~2.9]{Haragus2010} and \cite[Thm.~3.22]{Haragus2010} respectively). Undoing the coordinate transformation in \eqref{eq:K3_coord_change} yeilds the center manifold $\mathcal M_3^{a,+}$ described in the statement of the lemma. 
	
	It remains to show that the reduction function $\Psi_3^+$ has the asypmtotics in \eqref{eq:Psi_3}. First, we note that $\Psi_3^+(r_3,\eps_3)$ is spatially independent; this is a direct consequence of the fact that the variables $r_3$ and $\eps_3$ are spatially independent. It follows that $\Psi_3^+$ satisfies the ODE system
	\begin{equation}
		\label{eq:K3_ODEs}
		\begin{split}
			{\Psi^+_3}' &= \left( 1 -\frac{\eps_3}{2} \right) \Psi^+_3 - {\Psi^+_3}^3 + r_3 \left( \lambda \eps_3 + \mathcal R_3 (\Psi^+_3, r_3, \eps_3) \right) , \\
			r_3' &= \frac{1}{2} r_3 \eps_3 , \\
			\eps_3' &= - 2 \eps_3^2 .
		\end{split}
	\end{equation}
	Substituting a power series ansatz for $\Psi_3(r_3,\eps_3)$ and equating like powers in the corresponding invariance equation leads to the asymptotics in \eqref{eq:Psi_3}.
\end{proof}

For a graphical representation we refer again to Figure \ref{fig:center_manifolds}. Note that only one of the two center manifolds is visible in Figure \ref{fig:center_manifolds}, due to the fact that we plot the norm of $u_3$ on the vertical axis (which is necessarily positive). Lemma \ref{lem:K3_center_manifold} implies that solutions of system \eqref{eq:K3_eqns} with initial conditions in $\mathcal V_3^\pm$ are strongly attracted to the local center manifold $\mathcal M_3^{a,\pm}$, after which they track solutions of a $2$-dimensional system which in each case is governed by the planar ODE system
\[
\begin{split}
	r_3' &= \frac{1}{2} r_3 \eps_3 , \\
	\eps_3' &= - 2 \eps_3^2 .
\end{split}
\]
As in the $K_1$ analysis, this system can be solved directly, and the center manifolds $\mathcal M_3^{a,\pm}$ can be identified with the center manifolds which appear in \cite{Krupa2001c} (this time in the chart $K_3$ analysis), except that in our case, they are embedded and attracting within the phase space of the PDE system \eqref{eq:K3_eqns}.

\begin{remark}
	\label{rem:CM_restriction_pitchfork}
	One can show that $\mathcal M_3^{a,\pm}|_{\eps_3 = 0} = \mathcal P_3^\pm$, where $\mathcal P_3^{a,\pm}$ are the critical manifolds identified in \eqref{eq:P3} which coincide with the critical manifolds $\mathcal S_a^\pm \times \{0\}$ after blow-down on $\{r_3 > 0\}$ (see Figure \ref{fig:center_manifolds}). Using this fact together with the fact that (i) the blow-up transformation is a diffeomorphism on $\{r_3 > 0\}$, and (ii) the formal series expansions for the center manifold $\mathcal M_3^{a,\pm}$ and the slow manifold $\mathcal S_{a}^\pm$ are unique, it follows that $\mathcal M_3^{a,\pm}|_{r_3 > 0}$ coincides with the slow manifold $\mathcal S_{a}^\pm \times \{0\}$ after blowing down. A similar property is true in chart $K_1$, however, we choose to emphasise the point here instead because it will play a more important role in the proof.
\end{remark}

We can use Lemma \ref{lem:K3_center_manifold} to describe the local transition maps $\pi_3^\pm : \Sigma_3^{\textup{in},\pm} \to \Sigma_3^{\textup{out},\pm}$ induced by the forward evolution of initial conditions in
\[
\Sigma_3^{\textup{in},\pm} := \left\{(u_3, r_3, \delta) : \| u_3 \mp 1 \|_{\widetilde{\mathcal Z}} \leq \beta, r_3 \in [0, \nu] \right\}
\]
up to
\[
\Sigma_3^{\textup{out},\pm} := \left\{(u_3, \nu, \eps_3) : \| u_3 \mp 1 \|_{\widetilde{\mathcal Z}} \leq \beta, \eps_3 \in [0, \delta] \right\} ,
\]
see Figure \ref{fig:center_manifolds}. 

\begin{proposition}
	\label{prop:K3_dynamics}
	Fix $\lambda \neq 0$. For fixed but sufficiently small $\beta, \delta, \nu > 0$, one of the following assertions are true:
	\begin{enumerate}
		\item[(i)] $\lambda > 0$. In this case, $\pi_3^+$ is well-defined and given by
		\[
		\pi_3 : 
		\begin{pmatrix}
			u_3 \\
			r_3 \\
			\delta
		\end{pmatrix} \mapsto 
		\begin{pmatrix}
			\Psi_3^+ \left(\nu, \delta (r_3 / \nu)^4 \right) + \mathcal O \left( \e^{-\gamma \nu^2 / 2 \delta r_3^4} \right) \\
			\nu \\
			\delta (r_3 / \nu)^4
		\end{pmatrix} ,
		\]
		where $\Psi_3^+$ is the function which defines the center manifold $\mathcal M_3^{a,+}$, and $\gamma > 0$ is the constant which appears in Lemma \ref{lem:K3_center_manifold}.
		\item[(ii)] $\lambda < 0$. In this case, $\pi_3^-$ is well-defined and given by
		\[
		\begin{pmatrix}
			u_3 \\
			r_3 \\
			\delta
		\end{pmatrix} \mapsto 
		\begin{pmatrix}
			\Psi_3^- \left(\nu, \delta (r_3 / \nu)^4 \right) + \mathcal O \left( \e^{-\gamma \nu^4 / 2 \delta r_3^4} \right) \\
			\nu \\
			\delta (r_3 / \nu)^4
		\end{pmatrix} ,
		\]
		where $\Psi_3^-$ is the function which defines the center manifold $\mathcal M_3^{a,-}$, and $\gamma > 0$ is the constant which appears in Lemma \ref{lem:K3_center_manifold}.
	\end{enumerate}
	
\end{proposition}

\begin{proof}
	We consider the case with $\lambda > 0$. The case $\lambda < 0$ is analogous. The ODEs for $(r_3, \eps_3)$ decouple, and can be solved by direct integration. We also need an estimate for the transition time $T_3$, which can be determined using the constraint $r_3(t_3 = T_3) = \nu$. We obtain
	\[
	T_3 = \frac{1}{2 \delta} \left( \left( \frac{\nu}{r_3(0)} \right)^4 - 1 \right) .
	\]
	Using the attractivity of the center manifold $\mathcal M_3^{a,+}$, we obtain
	\[
	u_3(T_3) = \Psi_3 \left(r_3(T_3), \eps_3(T_3) \right) + \mathcal O \left( \e^{- \gamma T_3} \right) ,
	\]
	which yields the desired result after substituting the values for $T_3$, $r_3(T_3)$ and $\eps_3(T_3)$.
\end{proof}

\begin{remark}
	Based on comparisons with the ODE analysis in \cite{Krupa2001c}, one might expect the existence of a repelling local center manifold in a neighbourhood of the point $u_3 = r_3 = \eps_3 = 0$. However, direct calculation shows that $\sigma(\Lop_3) = (-\infty, 1]$, i.e.~that there is no spectral gap in this case due to the essential spectrum. This may lead to difficulties for approaches to the study of canard phenomena based on geometric blow-up and center manifold theory, in which a detailed understanding of the behaviour close to $u_3 = r_3 = \eps_3 = 0$ is necessary in order to describe the dynamics. We leave the challenging question of whether spatio-temporal canard phenomena can be described using geometric blow-up techniques of the kind developed herein to future work.
\end{remark}

\subsubsection{Dynamics in the rescaling chart $K_2$}

The equations in chart $K_2$ 
are obtained via simple rescalings, and given by
\begin{equation}
	\label{eq:K2_eqns}
	\begin{split}
		\partial_{t_2} u_2 &= \partial_{x_2}^2 u_2 + \mu_2 u_2 - u_2^3 + r_2 \lambda + r_2 \mathcal R_2(u_2, \mu_2, r_2) , \\
		\mu_2' &= 1 ,
	\end{split}
\end{equation}
where, as usual, we omit the $r_2' = 0$ equation and view $r_2 = \eps^{1/4} \ll 1$ as a perturbation parameter, $\mathcal R_2(u_2, \mu_2, r_2) := r_2^{-4} \mathcal R (r_2 u_2, r_2^2 \mu_2, r_2^4) = \mathcal O(u_2^4, r_2 u_2, r_2^2 \mu_2, r_2^4)$, we have rescaled space and time according to
\[
\frac{t_2}{2} = r_2^2 , \qquad \frac{x_2}{x} = r_2 ,
\]
and the prime now denotes differentiation with respect to $t_2$. We are interested in the evolution of solutions with initial conditions in
\[
\Sigma_2^{\textup{in}} := \left\{ (u_2, -\Omega, r_2) : \| u_2 \|_{\widetilde{\mathcal Z}} \leq \beta_-, r_2 \in [0,\varrho] \right\} ,
\]
where $\Omega > 0$ and $\beta_-, \varrho > 0$ are large respectively small positive constants, up to an exit section
\[
\Sigma_2^{\textup{out}, \pm} := \left\{ (u_2, \Omega, r_2) : \| u_2 \mp \sqrt{\Omega} \|_{\widetilde{\mathcal Z}} \leq \beta_+, r_2 \in [0,\varrho] \right\} ,
\]
where $\beta_+ > 0$ is another small but fixed positive constant. Notice that $\kappa_{12}(\Sigma_1^{\textup{out}}) \subseteq \Sigma_2^{\textup{in}}$ if we choose $\Omega = 1 / \sqrt \delta$, $\varrho = \delta^{1/4} \nu$, and $\beta > 0$ sufficiently small; this follows from direct calculations and an application of Lemma \ref{lem:norm_equivalence} in Appendix \ref{app:wieghted_spaces}. Similarly, $\Sigma_2^{\textup{out},\pm} \subseteq \kappa_{23}^{-1}(\Sigma_3^{\textup{in},\pm})$ if $\Omega = 1 / \sqrt \delta$, $\varrho = \delta^{1/4} \nu$, and $\beta_+ > 0$ is chosen sufficiently small. Finally, note that the time taken for solutions to move from $\Sigma_2^{\textup{in}}$ to the value $\mu_2(t_2) = \Omega$ corresponding to the sections $\Sigma_2^{\textup{out},\pm}$, assuming that they exist on this interval, is finite and given by $T_2 = 2 \Omega$ (this follows from $\dot \mu_2 = 1$). In the following result, we denote the action of the change of coordinate maps $\kappa_{ij}$ defined in \eqref{eq:kappa_ij} on the $u_i$-component by $\tilde \kappa_{ij}$.

\begin{lemma}
	\label{lem:K2_connection_pitchfork}
	For $\varrho > 0$ fixed but sufficiently small, there exists a $\mathcal C^k$-function $\Psi_2 : \R \times [0,\varrho] \to \R$ such that
	\begin{enumerate}
		\item[(i)] $\Psi_2(\mu_2, r_2) = \tilde \kappa_{12} \left( \Psi_1(r_1, \eps_1) \right)$ for all $\mu_2 < 0$;
		\item[(ii)] If $\lambda > 0$, then $\Psi_2(\mu_2, r_2) = \tilde \kappa_{23}^{-1} \left( \Psi_3^+(r_3, \eps_3) \right)$ for all $\mu_2 > 0$;
		\item[(iii)] If $\lambda < 0$, then $\Psi_2(\mu_2, r_2) = \tilde \kappa_{23}^{-1} \left( \Psi_3^-(r_3, \eps_3) \right)$ for all $\mu_2 > 0$;
		\item[(iv)] The two-dimensional manifold
		\[
		\mathcal M_2^a := \left\{ \left(\Psi_2(\mu_2,r_2) , \mu_2, r_2 \right) : \mu_2 \in [-\Omega, \Omega], r_2 \in [0,\varrho] \right\}
		\]
		is locally invariant for system \eqref{eq:K2_eqns}.
	\end{enumerate}
\end{lemma}

\begin{proof}
	This follows directly from the analysis in \cite{Krupa2001c} if we define $\Psi_2$ using the expression in Assertion (i) and then extend it under the flow of the planar ODE system
	\[
	\begin{split}
		\Psi_2' &= \mu_2 \Psi_2 - \Psi_2^3 + r_2 \lambda + r_2 \mathcal R_2(\Psi_2, \mu_2, r_2) , \\
		\mu_2' &= 1 ,
	\end{split}
	\]
	which governs the spatially independent solutions of system \eqref{eq:K2_eqns}.
\end{proof}

\begin{figure}[t!]
	\centering
	\includegraphics[scale=0.35]{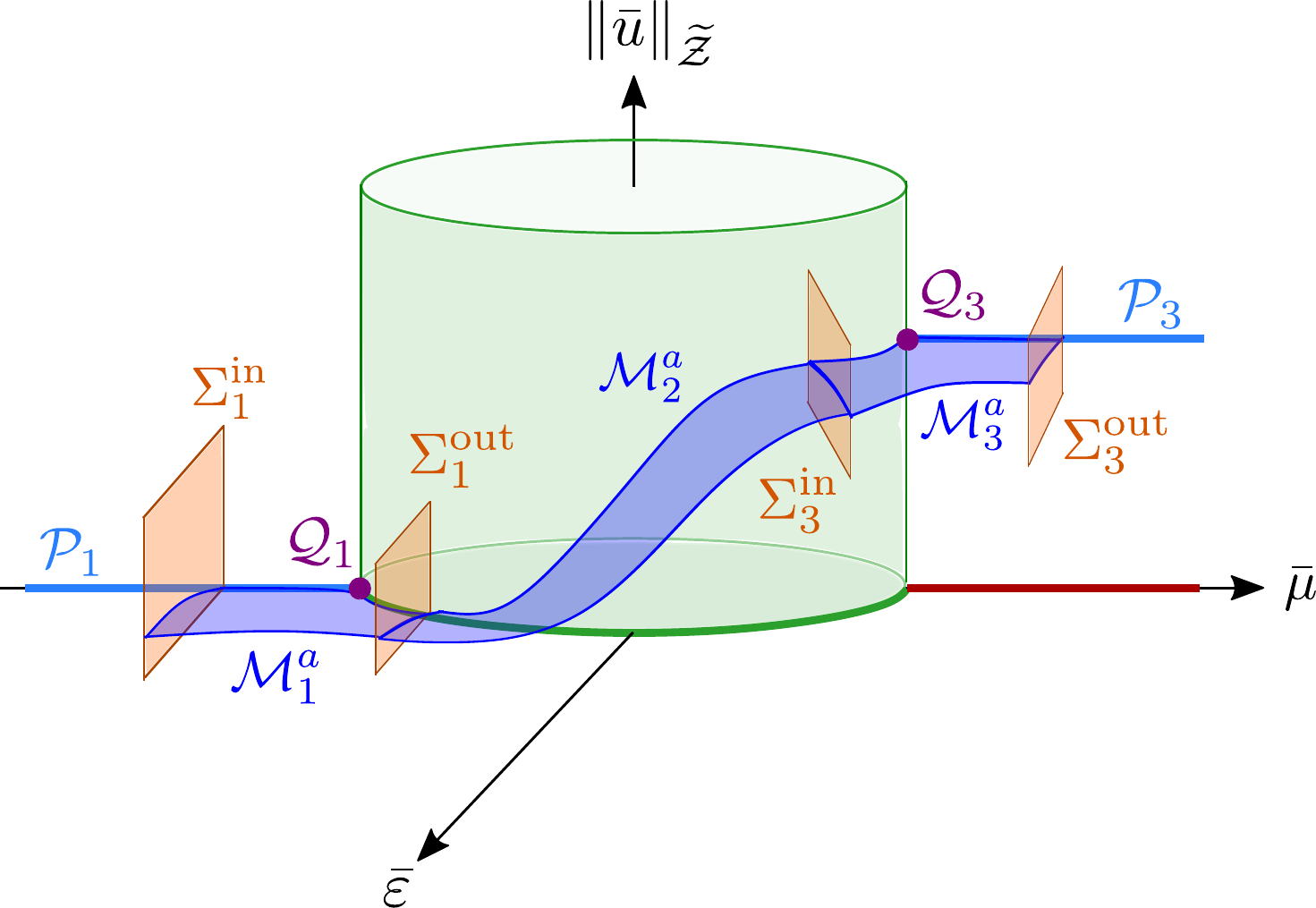}
	\caption{The $2$-dimensional manifold $\mathcal M_2^a$ described in Lemma \ref{lem:K2_connection_pitchfork} provides a connection between the local center manifolds $\mathcal M_1^a$ and $\mathcal M_3^a$ (which as in Figure \ref{fig:center_manifolds} denotes either $\mathcal M_3^{a,+}$ or $\mathcal M_3^{a,-}$, depending on the sign of $\lambda$). The entry and exit sections $\Sigma_2^{\textup{in}}$ and $\Sigma_2^{\textup{out},\pm}$ (not notated here), are super and subsets of $\Sigma_1^{\textup{out}}$ and $\Sigma_3^{\textup{in}}$ respectively.}
	\label{fig:connection}
\end{figure}

Lemma \ref{lem:K2_connection_pitchfork} establishes a connection between the center manifolds $\mathcal M_1^a$ and $\mathcal M_3^{a,\pm}$ identified in the entry and exit charts $K_1$ and $K_3$ respectively. In particular, $\mathcal M_1^a$ connects with $\mathcal M_3^{a,+}$ ($\mathcal M_3^{a,-}$) if $\lambda > 0$ ($\lambda < 0$). This is shown in Figure \ref{fig:connection}.

We now show that solutions of the PDE system \eqref{eq:K2_eqns} `track' solutions on $\mathcal M_2^a$. In particular, we describe the transition map $\pi_2^+ : \Sigma_2^{\textup{in}} \to \Sigma_2^{\textup{out},+}$ ($\pi_2^- : \Sigma_2^{\textup{in}} \to \Sigma_2^{\textup{out},-}$) induced by system \eqref{eq:K2_eqns} with $\lambda > 0$ ($\lambda < 0$).

\begin{proposition}
	\label{prop:K2_dynamics}
	The following assertions hold for sufficiently small but fixed $\beta_-, \varrho > 0$:
	\begin{enumerate}
		\item[(i)] If $\lambda > 0$, then the map $\pi_2^+$ is well-defined and given by
		\[
		\pi_2^+ : 
		\begin{pmatrix}
			u_2 \\
			-\Omega \\
			r_2
		\end{pmatrix}
		\mapsto 
		\begin{pmatrix}
			\Psi_2(\Omega, r_2) + \mathcal Q_2^+(u_2, r_2) \\
			\Omega \\
			r_2
		\end{pmatrix}
		\]
		where $\| \mathcal Q_2^+(u_2, r_2) \|_{\widetilde{\mathcal Z}} \leq C \| u_2 - \Psi_2(-\Omega, r_2) \|_{\widetilde{\mathcal Z}}$ for some constant $C > 0$.
		\item[(ii)] If $\lambda < 0$, then the map $\pi_2^-$ is well-defined and given by
		\[
		\pi_2^- : 
		\begin{pmatrix}
			u_2 \\
			-\Omega \\
			r_2
		\end{pmatrix}
		\mapsto 
		\begin{pmatrix}
			\Psi_2(\Omega, r_2) + \mathcal Q_2^-(u_2, r_2) \\
			\Omega \\
			r_2
		\end{pmatrix}
		\]
		where $\| \mathcal Q_2^-(u_2, r_2) \|_{\widetilde{\mathcal Z}} \leq C \| u_2 - \Psi_2(-\Omega, r_2) \|_{\widetilde{\mathcal Z}}$ for some constant $C > 0$.
	\end{enumerate}
\end{proposition}

\begin{proof}
	We fix $\lambda > 0$ and consider the map $\pi_2^+ : \Sigma_2^{\textup{in}} \to \Sigma_2^{\textup{out},+}$. The map $\pi_2^-$ with $\lambda < 0$ is analogous.
	
	The idea is to measure the distance between solutions of system \eqref{eq:K2_eqns} with initial conditions in $\Sigma_2^{\textup{in}}$ and solutions on $\mathcal M_2^a$ which, by Lemma \ref{lem:K2_connection_pitchfork}, intersects both $\Sigma_2^{\textup{in}}$ and $\Sigma_2^{\textup{out},+}$. In order to do so, we introduce an error function
	\[
	E(\cdot, t_2, r_2) := \tilde u_2(\cdot, t_2, r_2) - \Psi_2(\mu_2(t_2), r_2),
	\]
	where $\tilde u(\cdot, t_2, r_2)$ denotes a solution to system \eqref{eq:K2_eqns} on the time interval $[0,T_2]$, $\mu_2(t_2) = -\Omega + t_2$, and $\Psi_2$ is the function described in Lemma \ref{lem:K2_connection_pitchfork}. Differentiating with respect to $t_2$ leads to the following evolution equation for $E$:
	\[
	\partial_{t_2} E = \partial_{x_2}^2 E + \mathcal K(E, \Psi_2) E ,
	\]
	where $\mathcal K(E, \Psi_2) := \mu_2(t_2) - 3 \Psi_2^2 - 3 E \Psi_2 - E^2$ and we used the fact that $\Psi_2' = \mu_2(t_2) \Psi_2 - \Psi_2^3 + r_2 \lambda + r_2 \mathcal R_2(\Psi_2, \mu_2, r_2)$. 
	Applying the variation of constants formula yields
	\[
		E (\cdot, t_2) = \e^{t_2 \Lambda_2} E(\cdot, 0)
		+ \int_0^{t_2} \e^{(t_2 - s_2) \Lambda_2} \mathcal K \left( E(\cdot, s_2), \Psi_2(s_2) \right) E (\cdot, s_2) \dd s_2 ,
	\]
	where $(\e^{t_2 \Lambda_2})_{t_2 \geq 0}$ denotes the (heat) semigroup generated by $\Lambda_2 := \partial_{x_2}^2$. We note that $(\e^{t_2 \Lambda_2})_{t_2 \geq 0} \in \mathcal L(\widetilde{\mathcal Z})$ for each $t_2 > 0$; this is shown in Appendix \ref{app:wieghted_spaces}, Lemma \ref{lem:semigroup}.
	
	Let
	\[
	\tilde T_2 := \inf \left\{ T_2, \sup \left\{ t_2 > 0 : \| E(\cdot, t_2) \|_{\widetilde{\mathcal Z}} \leq C_0 \right\} \right\} ,
	\]
	where the size of the constant $C_0 > \| E(\cdot, 0) \|_{\widetilde{\mathcal Z}} \geq 0$ is determined by the size of $\beta_- > 0$ (since both $\tilde u_2(\cdot, 0, r_2)$ and $\Psi_2(-\Omega,r_2)$ should correspond to initial conditions in $\Sigma_2^{\textup{in}}$). The expression for $E(\cdot, t_2)$ given above is well-defined for all $t_2 \in [0, \tilde T_2]$. On this interval, $\Psi_2(t_2)$ is bounded (due to the analysis in \cite{Krupa2001c}), and $E(\cdot, t_2)$ is bounded (by the definition of $\tilde T_2$). 
	It follows from these observations and the fact that $\widetilde{\mathcal Z}$ is an algebra (see Lemma \ref{lem:algebra} in Appendix \ref{app:wieghted_spaces}) that
	\[
	\| \mathcal K \left( E(\cdot, t_2), \Psi_2(t_2) \right) \|_{\widetilde{\mathcal Z}} \leq C_1 ,
	\]
	where $C_1$ is a constant, for all $t_2 \in [0, \tilde T_2]$. Taking norms, applying a triangle inequality and appealing to the estimate above, we obtain
	\[
		\| E(\cdot, t_2) \|_{\widetilde{\mathcal Z}} \leq
		\| \e^{t_2 \Lambda_2} \|_{\widetilde{\mathcal Z} \to \widetilde{\mathcal Z}} \| E(\cdot, 0) \|_{\widetilde{\mathcal Z}} 
		+ C_1 \int_0^{t_2} \| \e^{(t_2 - s_2) \Lambda_2} \|_{\widetilde{\mathcal Z} \to \widetilde{\mathcal Z}} \| E(\cdot, s_2) \|_{\widetilde{\mathcal Z}} \dd s_2 ,
	\]
	which holds for all $t_2 \in [0, \tilde T_2]$. Applying the bound on operator norm $\| \e^{t_2 \Lambda_2} \|_{\widetilde{\mathcal Z} \to \widetilde{\mathcal Z}}$ provided in Appendix \ref{app:wieghted_spaces}, Lemma \ref{lem:semigroup}, we have that
	\begin{equation}
		\label{eq:operator_norm}
		\| \e^{t_2 \Lambda_2} \|_{\widetilde{\mathcal Z} \to \widetilde{\mathcal Z}} \leq 1
	\end{equation}
	for all $t_2 \in [0, \tilde T_2]$. Combining all of this with a Gr\"onwall inequality leads to
	\[
	\| E(\cdot, t_2) \|_{\widetilde{\mathcal Z}} \leq
	\e^{C_1 t_2} \| E(\cdot, 0) \|_{\widetilde{\mathcal Z}} 
	\]
	for all $t_2 \in [0, \tilde T_2]$. 
	Fixing $\beta_- > 0$ such that $\| E(\cdot, 0) \|_{\widetilde{\mathcal Z}} \leq C_0 \e^{- C_1 T_2}$ (which is possible since $T_2 > 0$ is finite) ensures that the corresponding bound for $E(\cdot, t_2)$ holds for all $t_2 \in [0,\tilde T_2]$ with $\tilde T_2 = T_2$.
	
	The preceding arguments show that
	\[
	u_2(\cdot, T_2) = \Psi_2(T_2) + E(\cdot, T_2) ,
	\]
	where $\| E(\cdot, T_2) \|_{\widetilde{\mathcal Z}} \leq C \| E(\cdot, 0) \|_{\widetilde{\mathcal Z}}$ for some constant $C$. From here, Assertion (i) 
	follows directly from Lemma \ref{lem:K2_connection_pitchfork} after equating $E(\cdot, T_2)$ with $\mathcal Q_2^+(u_2, r_2)$. 
	Assertion (ii) follows by analogous arguments.
\end{proof}

The key implication of Proposition \ref{prop:K2_dynamics} and its proof is that solutions to the PDE system \eqref{eq:K2_eqns} which have initial conditions that are close to the manifold $\mathcal M_2^a$, remain close to $\mathcal M_2^a$ up until their intersection with either $\Sigma_2^{\textup{out,+}}$ (if $\lambda > 0$) or $\Sigma_2^{\textup{out,-}}$ (if $\lambda < 0$). Note that ``closeness'' here is quantified in terms of the initial separation on $\Sigma_2^{\textup{in}}$.

\subsubsection{Proving Theorem \ref{thm:pitchfork}}
\label{subsec:pitchfork_proof}

We are now in a position to prove Theorem \ref{thm:pitchfork}. We prove Assertion (i) (case $\lambda > 0$); Assertion (ii) is similar. In order to do so, we show that the transition map \[
\pi^+ := \pi_3^+ \circ \kappa_{23} \circ \pi_2^+ \circ \kappa_{12} \circ \pi_1 : \Sigma_1^{\textup{in}} \to \Sigma_3^{\textup{out},+} ,
\]
is well-defined and given by
\begin{equation}
	\label{eq:pi_plus}
	\pi^+ : 
	\begin{pmatrix}
		u_1 \\
		\nu \\
		\eps_1
	\end{pmatrix}
	\mapsto
	\begin{pmatrix}
		\Psi_3^+(\nu, \eps_1) + \mathcal O(\e^{- \gamma / 2 \eps_1}) \\
		\nu \\
		\eps_1
	\end{pmatrix} .
\end{equation}
This will imply the desired result, since applying the blow-down transformation to the right-hand side -- specifically the relations $u = \nu u_3$ and $\eps_1 = \eps / \nu^4$ -- yields
\[
	u(x,T) = \nu \Psi_3(\nu, \eps / \nu^4) + \mathcal O( \e^{- \gamma \nu^4 / 2 \eps} ) 
	= \phi_+(\nu, \eps) + \mathcal O( \e^{- \gamma \rho^2 / 2 \eps} ) ,
\]
where the equality $\nu \Psi_3(\nu, \eps / \nu^4) = \phi_+(\nu, \eps)$ follows from the fact that $\mathcal M_3^{a,+}|_{\eps_3 = 0} = \mathcal P_3$, which coincides with $\mathcal S_a^+ \times \{0\}$ after blow-down on $\{r_3 > 0\}$, and the uniqueness of the formal series representation of the center and slow manifolds $\mathcal M_3^{a,+}$ and $\mathcal S_{a}^+$; see Remark \ref{rem:CM_restriction_pitchfork} for details.

Fix an initial condition $(u_1, \nu, \eps_1)^\transpose \in \Sigma_1^{\textup{in}}$. Applying Proposition \ref{prop:K1_dynamics}, the change of coordinates formula for $\kappa_{12}$ in \eqref{eq:kappa_ij}, Lemma \ref{lem:K2_connection_pitchfork} assertion (i) and the relation $\Omega = - \delta^{-1/2}$ yields
\[
	\kappa_{12} \circ \pi_1 : 
	\begin{pmatrix}
		u_1 \\
		\nu \\
		\eps_1
	\end{pmatrix}
	\mapsto
	\begin{pmatrix}
		\delta^{-1/4} \Psi_1(\nu (\eps_1 / \delta)^{1/4}, \delta) + \mathcal O(\e^{- \gamma / 2 \eps_1}) \\ - \delta^{-1/2} \\
		\nu \eps_1^{1/4}
	\end{pmatrix} 
	=
	\begin{pmatrix}
		\Psi_2(- \Omega, \nu \eps_1^{1/4}) + \mathcal O(\e^{- \gamma / 2 \eps_1}) \\ 
		- \Omega \\ 
		\nu \eps_1^{1/4}
	\end{pmatrix} \in \Sigma_2^{\textup{in}} .
\]
This shows that initial conditions in $\Sigma_1^{\textup{in}}$ are mapped to initial conditions in $\Sigma_2^{\textup{in}}$ that are $\mathcal O(\e^{- \gamma / 2 \eps_1})$-close to $\mathcal M_2^a$ in the $\widetilde{\mathcal Z}$ norm.

\begin{remark}
	Proposition \ref{prop:K1_dynamics} implies that initial conditions are $\mathcal O(\e^{- \gamma / 2 \eps_1})$-close to $\mathcal M_1^a \cap \Sigma_1^{\textup{out}}$ at the exit time $t_1 = T_1$. `Closeness' here is measured with respect to the $\widetilde{\mathcal Z}$ norm of functions which depend on the independent spatial variable $x_1$. This can be translated to $\mathcal O(\e^{- \gamma / 2 \eps_1})$-closeness in the $\widetilde{\mathcal Z}$ norm for functions of the independent spatial variable $x_2$ using Lemma \ref{lem:norm_equivalence} and the fact that
	\[
	x_1 = \delta^{-1/4} x_2
	\]
	for functions in $\Sigma_2^{\textup{in}} \subseteq \kappa_{12}(\Sigma_1^{\textup{out}})$. Similar observations allow us to make closeness statements of this kind when translating between charts $K_2$ and $K_3$; we refer again to Lemma \ref{lem:norm_equivalence} and the discussion which follows it in Appendix \ref{app:wieghted_spaces} for details.
\end{remark}

Combining the preceding arguments with Proposition \ref{prop:K2_dynamics}, it follows that
\[
\pi_2^+ \circ \kappa_{12} \circ \pi_1 : 
\begin{pmatrix}
	u_1 \\
	\nu \\
	\eps_1
\end{pmatrix}
\mapsto
\begin{pmatrix}
	\Psi_2(\Omega, \nu \eps_1^{1/4}) + \mathcal O(\e^{- \gamma / 2 \eps_1}) \\
	\Omega \\
	\nu \eps_1^{1/4} 
\end{pmatrix} \in \Sigma_2^{\textup{out},+} ,
\]
where we note that the error in the first component is exponentially small because
\[
\| \mathcal Q_2^+(\Psi_2(-\Omega, \nu \eps_1^{1/4}) + \mathcal O(\e^{- \gamma / 2 \eps_1}), \nu \eps_1) \|_{\widetilde{\mathcal Z}} =
\mathcal O(\e^{- \gamma / 2 \eps_1}) .
\]
Since we choose $\Sigma_2^{\textup{out},+}$ so that $\kappa_{23}(\Sigma_2^{\textup{out},+}) \subseteq \Sigma_3^{\textup{in},+}$, we can apply the change of coordinate formula for $\kappa_{23}$ in \eqref{eq:kappa_ij} and Proposition \ref{prop:K3_dynamics} directly to the preceding expression in order to show that the map $\pi^+$ is well-defined and given by \eqref{eq:pi_plus}, thereby proving the desired result. This concludes the proof.
%
%

\subsection{The transcritical case ($s = 2$)}
\label{sub:proof_transcritical}

We turn now to the proof of Theorem \ref{thm:transcritical}. Since the proof proceeds very similarly to the proof of Theorem \ref{thm:pitchfork} we present less details, focusing primarily on the aspects which differ. As before, we start with an independent analysis of the dynamics in charts $K_i$.

\begin{remark}
	In the following we shall reuse some of the notation that we used in our analysis of the pitchfork case in Section \ref{sub:proof_pitchfork} (for equilibria, steady state branches, center manifolds etc), in order to avoid excess sub and superscripts. The interpretation should be clear from the context.
\end{remark}

\subsubsection{Dynamics in the entry chart $K_1$}

After rewriting system \eqref{eq:main_system_extended} (now with $s = 2$) in chart $K_1$ coordinates and applying a state-dependent transformation of time and space which is defined via the relations
\[
\frac{\dd t_1}{\dd t} = r_1 , \qquad 
\frac{x_1}{x} = r_1^{1/2} ,
\]
we may write the equations in chart $K_1$ as
\begin{equation}
	\label{eq:K1_eqns_transcritical}
	\frac{\dd \bm{u}_1}{\dd t_1} = \Lop_1 \bm{u}_1 + \Nop_1(\bm{u}_1),
\end{equation}
where $\bm{u}_1 = (u_1, r_1, \eps_1)$, $\Lop_1$ is the linear operator defined via
\begin{equation}
	\label{eq:L1_transcritical}
	\Lop_1 \bm{u}_1 = 
	\begin{pmatrix}
		\partial_{x_1}^2 u_1 - u_1 + \lambda \eps_1 \\
		0 \\
		0
	\end{pmatrix}
\end{equation}
and $\Nop_1$ is a nonlinear operator defined via
\[
\Nop_1(\bm{u}_1) = 
\begin{pmatrix}
	\eps_1 u_1 - u_1^2 - \frac{1}{2} \eps_1 x_1 \partial_{x_1} u_1 + r_1 \mathcal R_1 (u_1, r_1, \eps_1) \\
	- r_1 \eps_1 \\
	2 \eps_1^2
\end{pmatrix},
\]
where $\mathcal R_1(u_1, r_1, \eps_1) := r_1^{-3} \mathcal R(r_1 u_1, -r_1, r_1^2 \eps_1)$, which is well-defined and $\mathcal O(u_1, \eps_1)$ due to the asymptotics in Lemma \ref{lem:normal_forms} Assertion (i).

System \eqref{eq:K1_eqns_transcritical} has a branch of spatially homogeneous steady states
\[
\mathcal P_1 = \left\{ (0, r_1, 0) : r_1 \in [0,\nu] \right\} 
\]
where $\nu > 0$ is a small constant, and we note that for $r_1 > 0$, $\mathcal P_1$ corresponds to the branch of steady states along $\{ (0,\mu, 0) : \mu < 0 \}$ in the extended system \eqref{eq:main_system_extended} under the blow-up map $\Phi$, viewed in chart $K_1$. Our aim is to apply center manifold theory in a neighbourhood of the point \SJ{$\mathcal Q_1 := (0,0,0) \in \mathcal P_1$}, which lies at the intersection of $\mathcal P_1$ with the blow-up manifold (where $r_1 = 0$).

\begin{lemma}
	\label{lem:K1_center_manifold_transcritical}
	There exists a neighbourhood $\mathcal U_1$ of $\mathcal Q_1$ in $\mathcal Z$ such that system \eqref{eq:K1_eqns_transcritical} has a $2$-dimensional local center manifold
	\[
	\mathcal M_1^a = \left\{ (u_1, r_1, \eps_1) \in \mathcal U_1 : u_1 = \Psi_1(r_1, \eps_1) \right\} ,
	\]
	which is locally invariant and tangent to $\mathcal E_0$ at $\mathcal Q_1$, for which the reduction function $\Psi_1 \in \mathcal C^k(\mathcal E_0, \widetilde{\mathcal Z})$ has asymptotics
	\begin{equation}
		\label{eq:Psi_1_transcritical}
		\Psi_1(r_1, \eps_1) = \lambda \eps_1 + \mathcal O( |(r_1, \eps_1) |^2 ) .
	\end{equation}
	Moreover, $\mathcal M_1^a$ is locally attracting: if $\bm{u}_1(0) \in \mathcal U_1$ and $\bm u_1(t_1; \bm{u}_1(0)) \in \mathcal U_1$ for all $t_1 > 0$, then there exists a `base function' $\tilde{\bm{u}}_1 \in \mathcal M_1^a \cap \mathcal U_1$ and a $\gamma > 0$ such that
	\[
	\bm{u}_1(t_1; \bm{u}_1(0)) = \bm{u}_1 (t_1; \tilde{\bm{u}}_1) + \mathcal O(\e^{- \gamma t_1}) 
	\]
	as $t_1 \to \infty$.
\end{lemma}

\begin{proof}
	The idea is to apply the center manifold theorem after a preliminary change of variables $\tilde u_1 = u_1 - \lambda_1 \eps_1$, which transforms system \eqref{eq:K1_eqns_transcritical} into
	\begin{equation}
		\label{eq:K1_eqns_transcritical_transformed}
		\frac{\dd \tilde{\bm{u}}_1}{\dd t_1} = \widetilde{\Lop}_1 \tilde{\bm{u}}_1 + \widetilde{\Nop}_1(\tilde{\bm{u}}_1) ,
	\end{equation}
	where $\tilde{\bm{u}}_1 = (\tilde u_1, r_1, \eps_1)$,
	\[
	\widetilde{\Lop}_1 \tilde{\bm{u}}_1 = 
	\begin{pmatrix}
		\partial_{x_1}^2 \tilde u_1 - \tilde u_1 \\
		0 \\
		0
	\end{pmatrix}
	\]
	and
	\[
	\widetilde{\Nop}_1(\tilde{\bm{u}}_1) = 
	\begin{pmatrix}
		- \frac{1}{2} \eps_1 x_1 \partial_{x_1} u_1 + \mathcal O( \tilde u_1^2, \tilde u_1 r_1, \tilde u_1 \eps_1, r_1^2, r_1 \eps_1, \eps_1^2 ) \\
		- r_1 \eps_1 \\
		2 \eps_1^2
	\end{pmatrix} .
	\]
	We have chosen not to write the entire first component of $\widetilde{\Nop}_1(\tilde{\bm{u}}_1)$ explicitly for convenience, since it will not play an important role in the proof.
	
	Notice that $\widetilde \Lop_1$ -- in contrast to $\Lop_1$ -- is diagonal. In fact, $\widetilde \Lop_1$ has precisely the same form as the linear operator $\Lop_1$ which arose in the pitchfork analysis in chart $K_1$, recall equation \eqref{eq:L1}. This allows us to apply the center manifold theorem to system \eqref{eq:K1_eqns_transcritical_transformed} immediately, since Hypothesis \ref{hyp:1} can be verified directly, and Hypotheses \ref{hyp:2} and \ref{hyp:3} follow from Lemmas \ref{lem:spectrum_K1} and \ref{lem:resolvent_K1} respectively. This yields a local center manifold of the form
	\[
	\widetilde{\mathcal M}_1^a = \left\{ (\tilde u_1, r_1, \eps_1) \in \widetilde{\mathcal U}_1 : \tilde u_1 = \tilde \Psi_1(r_1, \eps_1) \right\} ,
	\]
	where $\tilde \Psi_1(r_1, \eps_1) = \mathcal O( |(r_1, \eps_1) |^2 )$ and $\widetilde{\mathcal U}_1$ is a neighbourhood of $\bm{0}$ in $\mathcal Z$. Undoing the transformation $\tilde u_1 = u_1 - \lambda_1 \eps_1$ leads to the desired center manifold $\mathcal M_1^a$ and the asymptotics in \eqref{eq:Psi_1_transcritical}.
\end{proof}

We now use Lemma \ref{lem:K1_center_manifold_transcritical} in order to describe the transition map $\pi_1 : \Sigma_1^{\textup{in}} \to \Sigma_1^{\textup{out}}$ induced by the forward evolution of initial conditions in
\[
\Sigma_1^{\textup{in}} := \left\{(u_1, \rho, \eps_1) : \| u_1 \|_{\widetilde{\mathcal Z}} \leq \beta, \eps_1 \in [0, \delta] \right\}
\]
up to
\[
\Sigma_1^{\textup{out}} := \left\{(u_1, r_1, \delta) : \| u_1 \|_{\widetilde{\mathcal Z}} \leq \beta, r_1 \in [0, \rho] \right\} ,
\]
see Figure \ref{fig:center_manifolds}, which is also representative of the local geometry in the transcritical case. As in Section \ref{sub:proof_pitchfork}, we choose the constants defining $\Sigma_1^{\textup{in}}$ so that it contains the preimage of $\Delta^{\textup{in}} \times [0,\eps_0]$ under the blow-up map in chart $K_1$.

\begin{proposition}
	\label{prop:K1_dynamics_transcritical}
	The transition map $\pi_1$ is well-defined and given by
	\[
	\pi_1 : 
	\begin{pmatrix}
		u_1 \\
		\rho \\
		\eps_1
	\end{pmatrix} \mapsto 
	\begin{pmatrix}
		\Psi_1 \left(\rho (\eps_1 / \delta)^{1/2}, \delta \right) + \mathcal O(\e^{- \gamma / 2 \eps_1}) \\
		\rho (\eps_1 / \delta)^{1/2} \\
		\delta
	\end{pmatrix} ,
	\]
	where $\Psi_1$ is the function which defines the center manifold $\mathcal M_1^a$, and $\gamma > 0$ is the constant which appears in Lemma \ref{lem:K1_center_manifold_transcritical}.
\end{proposition}

\begin{proof}
	Similarly to the proof of Proposition \ref{prop:K1_dynamics}, $r_1(T_1) = \rho (\eps_1 / \delta)^{1/2}$ can be shown by direct integration, and the transition time $T_1$ can be determined using the constraint $\eps_1(t_1 = T_1) = \delta$. We obtain
	\[
	T_1 = \frac{1}{2 \eps_1(0)} - \frac{1}{2 \delta} .
	\]
	Using the attractivity of the center manifold $\mathcal M_1^a$, we obtain
	\[
	u_1(T_1) = \Psi_1 \left(r_1(T_1), \eps_1(T_1) \right) + \mathcal O \left( \e^{- \gamma T_1} \right) ,
	\]
	which yields the desired result after substituting the values for $T_1$, $r_1(T_1)$ and $\eps_1(T_1)$.
\end{proof}

Proposition \ref{prop:K1_dynamics_transcritical} shows that solutions of system \eqref{eq:K1_eqns_transcritical} are exponentially close to the spatially independent solutions in $\mathcal M_1^a \cap \Sigma_1^{\textup{out}}$ when they leave a neighbourhood of $\mathcal Q_1$ via $\Sigma_1^{\textup{out}}$.

\subsubsection{Dynamics in the exit chart $K_3$}

We now consider the dynamics in the exit chart $K_3$. Rewriting the equations in chart $K_3$ coordinates and applying a state-dependent transformation of time and space defined via
\[
\frac{\dd t_3}{\dd t} = r_3, \qquad \frac{x_3}{x} = r_3^{1/2} ,
\]
leads to the system
\begin{equation}
	\label{eq:K3_eqns_transcritical}
	\frac{\dd \bm{u}_3}{\dd t_3} = \Lop_3 \bm{u}_3 + \Nop_3(\bm{u}_3),
\end{equation}
where $\bm{u}_3 = (u_3, r_3, \eps_3)^\transpose \in \mathcal Z$, and $\Lop_3$, $\Nop_3$ are linear respectively nonlinear operators defined by
\[
\Lop_3 \bm{u}_3 = 
\begin{pmatrix}
	\partial_{x_3}^2 u_3 + u_3 + \lambda \eps_3 \\
	0 \\
	0
\end{pmatrix}
\]
and
\[
	N_3(\bm{u}_3) =
	\begin{pmatrix}
		- \eps_3 u_3 - u_3^2 - \frac{1}{2} \eps_3 x_3 \partial_{x_3} u_3 + r_3 \mathcal R_3(u_3, r_3, \eps_3) \\
		r_3 \eps_3 \\
		- 2 \eps_3^2
	\end{pmatrix},
\]
where $\mathcal R_3(u_3, r_3, \eps_3) = r_3^{-3} \mathcal R_3(r_3 u_3, r_3, r_3^2 \eps_3)$, which is well-defined and $\mathcal O(u_3, \eps_3)$ due to the asymptotics in Lemma \ref{lem:normal_forms}, Assertion (i).

System \eqref{eq:K3_eqns_transcritical} has a branch of spatially homogeneous steady states along
\[
\mathcal P_3 = \left\{ (1 + \mathcal O(r_3), r_3, 0) : r_3 [0,\nu] \right\} 
\]
which, for $r_3 > 0$, corresponds to the critical manifold $\{ (\phi(\mu,0),\mu,0) : \mu > 0 \}$ (recall \eqref{eq:slow_manifolds_transcritical}) in the extended system \eqref{eq:main_system_extended} under the blow-up map $\Phi$, viewed in chart $K_3$. We want to apply center manifold theory in a neighbourhood of the point $\mathcal Q_3 := (1, 0, 0) \in \mathcal P_3$, which lies at the intersection of $\mathcal P_3$ with the blow-up manifold (where $r_3 = 0$).

\begin{lemma}
	\label{lem:K3_center_manifold_transcritical}
	There exists a neighbourhood $\mathcal U_3$ of $\mathcal Q_3$ in $\mathcal Z$ such that system \eqref{eq:K3_eqns_transcritical} has a $2$-dimensional local center manifold
	\[
	\mathcal M_3^a = \left\{ (u_3, r_3, \eps_3) \in \mathcal U_3 : u_3 = \Psi_3(r_3, \eps_3) \right\} ,
	\]
	which is locally invariant and tangent to $\mathcal E_0$ at $\mathcal Q_3$, for which the reduction function $\Psi_3 \in \mathcal C^k(\mathcal E_0, \widetilde{\mathcal Z})$ has asymptotics
	\begin{equation}
		\label{eq:Psi_3_transcritical}
		\Psi_3(r_3, \eps_3) = 1 - (1 - \lambda) \eps_3 + \mathcal R_3(1,0,0) r_3 + \mathcal O( |(r_3, \eps_3) |^2 ) .
	\end{equation}
	Moreover, $\mathcal M_3^a$ is locally attracting: if $\bm{u}_3(0) \in \mathcal U_3$ and $\bm u_3(t_3; \bm{u}_3(0)) \in \mathcal U_3$ for all $t_3 > 0$, then there exists a `base function' $\tilde{\bm{u}}_3 \in \mathcal M_3^a \cap \mathcal U_3$ and a $\gamma > 0$ such that
	\[
	\bm{u}_3(t_3; \bm{u}_3(0)) = \bm{u}_3 (t_3; \tilde{\bm{u}}_3) + \mathcal O(\e^{- \gamma t_3}) 
	\]
	as $t_3 \to \infty$.
\end{lemma}

\begin{proof}
	The proof proceeds analogously to the proof of Lemma \ref{lem:K1_center_manifold_transcritical} after a preliminary transformation $\tilde u_3 = - 1 + u_3 + (1 - \lambda) \eps_3 - \mathcal R_3(1,0,0) r_3$, which (i) translates $\mathcal Q_3$ to the origin and (ii) diagonalises the linear part so that Lemma's \ref{lem:spectrum_K1} and \ref{lem:resolvent_K1} can again be applied in order to verify the hypotheses of the center manifold theorem.
\end{proof}

\begin{remark}
	\label{rem:CM_restriction_transcritical}
	An analogue to Remark \ref{rem:CM_restriction_pitchfork} applies. In particular, $\mathcal M_3^a|_{\eps_3 = 0} = \mathcal P_3$, and $\mathcal M_3^a|_{r_3 > 0}$ coincides locally with the slow manifold $\mathcal S_{a}^+$ after blowing down.
\end{remark}

Using Lemma \ref{lem:K3_center_manifold_transcritical}, we can describe the local transition map $\pi_3 : \Sigma_3^{\textup{in}} \to \Sigma_3^{\textup{out}}$ induced by the forward evolution of initial conditions in
\[
\Sigma_3^{\textup{in}} := \left\{(u_3, r_3, \delta) : \| u_3 - 1 \|_{\widetilde{\mathcal Z}} \leq \beta, r_3 \in [0, \rho] \right\}
\]
up to
\[
\Sigma_3^{\textup{out}} := \left\{(u_3, \rho, \eps_3) : \| u_3 - 1 \|_{\widetilde{\mathcal Z}} \leq \beta, \eps_3 \in [0, \delta] \right\} ,
\]
see Figure \ref{fig:center_manifolds}. 

\begin{proposition}
	\label{prop:K3_dynamics_transcritical}
	For fixed but sufficiently small $\beta, \delta, \rho > 0$, the transition map $\pi_3$ is well-defined and given by
	\[
	\pi_3 : 
	\begin{pmatrix}
		u_3 \\
		r_3 \\
		\delta
	\end{pmatrix} \mapsto 
	\begin{pmatrix}
		\Psi_3^+ \left(\rho, \delta (r_3 / \rho)^2 \right) + \mathcal O \left( \e^{-\gamma \rho^2 / 2 \delta r_3^2} \right) \\
		\rho \\
		\delta (r_3 / \rho)^2
	\end{pmatrix} ,
	\]
	where $\Psi_3$ is the function which defines the center manifold $\mathcal M_3^a$, and $\gamma > 0$ is the constant which appears in Lemma \ref{lem:K3_center_manifold_transcritical}.	
\end{proposition}

\begin{proof}
	This follows from arguments analogous to those presented in the proof of Proposition \ref{prop:K1_dynamics}, \ref{prop:K3_dynamics} and \ref{prop:K1_dynamics_transcritical}, so we omit the details for brevity.
\end{proof}

\subsubsection{Dynamics in the rescaling chart $K_2$}

In chart $K_2$ the equations are obtained via simple rescalings. They are given by
\begin{equation}
	\label{eq:K2_eqns_transcritical}
	\begin{split}
		\partial_{t_2} u_2 &= \partial_{x_2}^2 u_2 + \mu_2 u_2 - u_2^2 + \lambda + r_2 \mathcal R_2(u_2, \mu_2, r_2) , \\
		\mu_2' &= 1 ,
	\end{split}
\end{equation}
where we omit the trivial equation ($r_2' = 0$) and view $r_2 = \eps^{1/2} \ll 1$ as a perturbation parameter, $\mathcal R_2(u_2, \mu_2, r_2) := r_2^{-3} \mathcal R(r_2 u_2, r_2 \mu_2, r_2^2) = \mathcal O(u_2, \mu_2, r_2)$, we have rescaled space and time according to
\[
\frac{t_2}{2} = r_2 , \qquad \frac{x_2}{x} = r_2^{1/2} ,
\]
and the prime now denotes differentiation with respect to $t_2$. We are interested in the evolution of solutions with initial conditions in
\[
\Sigma_2^{\textup{in}} := \left\{ (u_2, -\Omega, r_2) : \| u_2 \|_{\widetilde{\mathcal Z}} \leq \beta_-, r_2 \in [0,\varrho] \right\} ,
\]
where $\Omega > 0$ and $\beta_-, \varrho > 0$ are large respectively small positive constants, up to an exit section
\[
\Sigma_2^{\textup{out}} := \left\{ (u_2, \Omega, r_2) : \| u_2 \|_{\widetilde{\mathcal Z}} \leq \beta_+, r_2 \in [0,\varrho] \right\} ,
\]
where $\beta_+$ is another small positive constant. We choose $\Omega, \beta_\pm, \varrho$ so that $\kappa_{12}(\Sigma_1^{\textup{out}}) \subseteq \Sigma_2^{\textup{in}}$ and $\Sigma_2^{\textup{out}} \subseteq \kappa_{23}^{-1}(\Sigma_3^{\textup{in}})$, and note that the time taken for solutions to move from $\Sigma_2^{\textup{in}}$ to $\Sigma_2^{\textup{out}}$, assuming that they exist on this interval, is finite and given by $T_2 = 2 \Omega$. As in Lemma \ref{lem:K2_connection_pitchfork}, we denote the action of the change of coordinate maps $\kappa_{ij}$ defined in \eqref{eq:kappa_ij} on the $u_i$-component by $\tilde \kappa_{ij}$.

\begin{lemma}
	\label{lem:K2_connection_transcritical}
	For $\varrho > 0$ fixed but sufficiently small, there exists a $\mathcal C^k$-function $\Psi_2 : \R \times [0,\varrho] \to \R$ such that
	\begin{enumerate}
		\item $\Psi_2(\mu_2, r_2) = \tilde \kappa_{12} \left( \Psi_1(r_1, \eps_1) \right)$ for all $\mu_2 < 0$;
		\item $\Psi_2(\mu_2, r_2) = \tilde \kappa_{23}^{-1} \left( \Psi_3(r_3, \eps_3) \right)$ for all $\mu_2 > 0$;
		\item The two-dimensional manifold
		\[
		\mathcal M_2^a := \left\{ \left(\Psi_2(\mu_2,r_2) , \mu_2, r_2 \right) : \mu_2 \in [-\Omega, \Omega], r_2 \in [0,\varrho] \right\}
		\]
		is locally invariant for system \eqref{eq:K2_eqns_transcritical}.
	\end{enumerate}
\end{lemma}

\begin{proof}
	This follows directly from the analysis in \cite{Krupa2001c} if we define $\Psi_2$ using the expression in Assertion 1 and then extend it under the flow of the planar ODE system
	\[
	\begin{split}
		\Psi_2' &= \mu_2 \Psi_2 - \Psi_2^2 + \lambda + r_2 \mathcal R_2(\Psi_2, \mu_2, r_2) , \\
		\mu_2' &= 1 ,
	\end{split}
	\]
	which governs the spatially independent solutions of system \eqref{eq:K2_eqns_transcritical}.
\end{proof}

Lemma \ref{lem:K2_connection_transcritical} establishes a connection between the center manifolds $\mathcal M_1^a$ and $\mathcal M_3^a$ identified in the entry and exit charts $K_1$ and $K_3$ respectively, under the condition that $\lambda > 0$. We now show that solutions of the PDE system \eqref{eq:K2_eqns_transcritical} `track' solutions on $\mathcal M_2^a$. This is formulated in terms of the transition map $\pi_2 : \Sigma_2^{\textup{in}} \to \Sigma_2^\textup{out}$ induced by system \eqref{eq:K2_eqns_transcritical} with $\lambda > 0$.

\begin{proposition}
	\label{prop:K2_dynamics_transcritical}
	Fix $\lambda > 0$. For $\beta_-, R > 0$ sufficiently small but fixed, the map $\pi_2$ is well-defined and given by
	\[
	\pi_2 : 
	\begin{pmatrix}
		u_2 \\
		-\Omega \\
		r_2
	\end{pmatrix}
	\mapsto 
	\begin{pmatrix}
		\Psi_2(\Omega, r_2) + \mathcal Q_2(u_2, r_2) \\
		\Omega \\
		r_2
	\end{pmatrix}
	\]
	where $\| \mathcal Q_2(u_2,r_2) \|_{\widetilde{\mathcal Z}} \leq C \| u_2 - \Psi_2(-\Omega, r_2) \|_{\widetilde{\mathcal Z}}$ for some constant $C > 0$.
\end{proposition}

\begin{proof}
	Following the arguments given in the proof of Proposition \ref{prop:K2_dynamics}, we introduce an error function
	\[
	E(\cdot, t_2, r_2) := \tilde u_2(\cdot, t_2, r_2) - \Psi_2(\mu_2(t_2), r_2),
	\]
	where $\tilde u(\cdot, t_2, r_2)$ denotes a solution to system \eqref{eq:K2_eqns_transcritical} on the time interval $[0,T_2]$, $\mu_2(t_2) = -\Omega + t_2$, and $\Psi_2$ is the function described in Lemma \ref{lem:K2_connection_transcritical}. Differentiating with respect to $t_2$ leads to the following evolution equation for $E$:
	\[
	\partial_{t_2} E = \partial_{x_2}^2 E + \mathcal K(E, \Psi_2) E ,
	\]
	where $\mathcal K(E, \Psi_2) := \mu_2(t_2) - 2 \Psi_2 - E + r_2 \widetilde{\mathcal R}_2(E, \Psi_2, \mu_2, r_2)$, where
	\[
	\widetilde{\mathcal R}_2(E, \Psi_2, \mu_2, r_2) := \frac{\mathcal R_2(E + \Psi_2, \mu_2, r_2) - \mathcal R_2(\Psi_2, \mu_2, r_2)}{E} ,
	\]
	which by Taylor's theorem is well-defined as $E \to 0$.
	
	Let 
	\[
	\tilde T_2 := \inf \left\{ T_2, \sup \left\{ t_2 \geq 0 : \| E(\cdot, t_2) \|_{\widetilde{\mathcal Z}} \leq C_2 \right\} \right\} ,
	\]
	where the constant $C_2 \geq \| E(\cdot, 0) \|_{\widetilde{\mathcal Z}} \geq 0$ is fixed but small enough to ensure that $\tilde u_2(\cdot, 0, r_2)$ and $\Psi_2(- \Omega, r_2)$ correspond to initial conditions in $\Sigma_2^{\textup{in}}$, as well as the validity of the Taylor expansion about $E = 0$ which allows us to define the function $\widetilde R_2$ above. On $t_2 \in [0, \tilde T_2]$, the variation of constants formula yields
	\[
		E (\cdot, t_2) = \e^{t_2 \Lambda_2} E(\cdot, 0)
		+ \int_0^{t_2} \e^{(t_2 - s_2) \Lambda_2} \mathcal K \left( E(\cdot, s_2), \Psi_2(s_2) \right) E (\cdot, s_2) \dd s_2 ,
	\]
	where $(\e^{t_2 \Lambda_2})_{t_2 \geq 0}$ denotes the (heat) semigroup generated by $\Lambda_2 := \partial_{x_2}^2$. Note that for notational simplicity, we omit the dependence on $r_2$ in our notation.
	
	We want a bound for $\| E(\cdot, t_2) \|_{\widetilde{\mathcal Z}}$. In order to obtain one, notice first that $\Psi_2(t_2)$ is bounded (due to the analysis in \cite{Krupa2001c}), and that $E(\cdot, t_2)$ is bounded (by the definition of $\tilde T_2$). Therefore,
	\[
	\| \mathcal K \left( \tilde E(\cdot, t_2), \Psi_2(t_2) \right) \|_{\widetilde{\mathcal Z}} \leq C_1
	\]
	for some constant $C_1 > 0$, for all $t_2 \in [0,\tilde T_2]$ (we also use the fact that $\widetilde{\mathcal Z}$ is an algebra -- we refer again to Lemma \ref{lem:algebra} in Appendix \ref{app:wieghted_spaces} -- to control the higher order terms). Moreover, the operator norm on the semigroup obeys the same bound as in \eqref{eq:operator_norm} for all $t_2 \in [0,\tilde T_2]$. Using a triangle inequality together with both of these bounds and a Gr\"onwall inequality, we obtain the estimate
	\[
	\| E(\cdot, t_2) \|_{\widetilde{\mathcal Z}} \leq
	\e^{C_1 t_2} \| E(\cdot, 0) \|_{\widetilde{\mathcal Z}} 
	\]
	for all $t_2 \in [0, \tilde T_2]$. Fixing $\beta_- > 0$ (and therefore also $C_0$) such that
	\[
	\| E(\cdot, 0) \|_{\widetilde{\mathcal Z}} \leq C_0 \e^{- C_1 C_2 T_2} 
	\]
	ensures that the corresponding bound for $E(\cdot, t_2)$ holds for all $t_2 \in [0,\tilde T_2]$ with $\tilde T_2 = T_2$.
	
	The preceding arguments show that
	\[
	u_2(\cdot, T_2) = \Psi_2(T_2) + E(\cdot, T_2) ,
	\]
	where $\| E(\cdot, T_2) \|_{\widetilde{\mathcal Z}} \leq C \| E(\cdot, 0) \|_{\widetilde{\mathcal Z}}$. Equating $\mathcal Q_2(u_2, r_2)$ with $E(\cdot, T_2)$ concludes the proof.
\end{proof}

The key implication of Proposition \ref{prop:K2_dynamics_transcritical} and its proof is that solutions to the PDE system \eqref{eq:K2_eqns_transcritical}, with fixed $\lambda > 0$ and initial conditions that are close to the manifold $\mathcal M_2^a$, remain close to $\mathcal M_2^a$ up until their intersection with $\Sigma_2^{\textup{out}}$.

\subsubsection{Proving Theorem \ref{thm:transcritical}}

We are now in a position to prove Theorem \ref{thm:transcritical}. Using arguments which are very similar to those used in the proof of Theorem \ref{thm:pitchfork} in Section \ref{subsec:pitchfork_proof}, one can show that Propositions \ref{prop:K1_dynamics_transcritical}, \ref{prop:K3_dynamics_transcritical} and \ref{prop:K2_dynamics_transcritical} imply that the transition map
\[
\pi := \pi_3 \circ \kappa_{23} \circ \pi_2 \circ \kappa_{12} \circ \pi_1 : \Sigma_1^{\textup{in}} \to \Sigma_3^{\textup{out}} ,
\]
is well-defined and given by
\[
\pi : 
\begin{pmatrix}
	u_1 \\
	\rho \\
	\eps_1
\end{pmatrix}
\mapsto
\begin{pmatrix}
	\Psi_3(\rho, \eps_1) + \mathcal O(\e^{- \gamma / 2 \eps_1}) \\
	\rho \\
	\eps_1
\end{pmatrix} .
\]
Applying the blow-down transformation to the right-hand side -- specifically the relations $u = \rho u_3$ and $\eps_1 = \eps / \rho^2$ -- yields
\[
	u(x,T) = \rho \Psi_3(\rho, \eps / \rho^2) + \mathcal O( \e^{- \gamma \rho^2 / 2 \eps} )
	= \phi(\rho, \eps) + \mathcal O( \e^{- \gamma \rho^2 / 2 \eps} ) ,
\]
where the local coincidence of the center manifold $\mathcal M_3^a|_{r_3 > 0}$ and the slow manifold $\mathcal S_{a}^+ \times \{0\}$, as described in Remark \ref{rem:CM_restriction_transcritical}, is used in order to obtain the equality $\rho \Psi_3(\rho, \eps / \rho^2) = \phi(\rho,\eps)$. This concludes the proof.

%

\section{Summary and outlook}
\label{sec:summary_and_outlook}

The primary aim of this article has been to show that the geometric blow-up method can be used to obtain rigorous and geometrically informative information about the dynamics close to dynamic bifurcations in scalar reaction-diffusion equations in the general form \eqref{eq:main_general}, which are posed on the entire real line $x \in \R$. We considered two particular cases, corresponding to a dynamic transcritical or pitchfork bifurcation in the reaction dynamics. After deriving local normal forms in each case in Lemma \ref{lem:normal_forms}, we proved a number of exchange of stability properties, as described in Theorems \ref{thm:transcritical} and \ref{thm:pitchfork}, which are dynamically very similar to the exchange of stability properties that were proven for the corresponding problems on bounded spatial domains using upper and lower solutions in \cite{Butuzov2000}. The results in Theorems \ref{thm:transcritical} and \ref{thm:pitchfork} are of independent interest -- we refer back to the list of implications a)-f) in the discussion towards the end of Section \ref{sec:main_results} -- however, we consider the primary contribution of this article to be the method of proof itself. In light of the notable success of approaches to the study of dynamic bifurcations and singularities in fast-slow systems in the ODE setting (we refer again to \cite{Jardon2019b} and the many references therein), we are hopeful that the geometric blow-up based approach for scalar reaction-diffusion systems developed herein may provide a blueprint for a rigorous, geometrically informative and constructive approach to the study of dynamic bifurcations in PDEs more generally.

Relative to the recent adaptations and applications of the geometric blow-up method in the PDE setting in \cite{Engel2024,Engel2020,Jelbart2022,Jelbart2023,Zacharis2023}, the problems considered in this work are `dynamically simpler' because of the simple form of the slow equation $\dot \mu = \eps$, and because we do not consider more complicated spatial dynamics associated with e.g.~pattern formation. As a result, the proofs that we presented in Section \ref{sec:proofs} can more easily be compared and contrasted with the geometric blow-up based proofs for the corresponding ODE problems presented in \cite{Krupa2001c}. We would like to reiterate the following points in particular:
\begin{itemize}
	\item[(a)] The blow-up transformation $\Phi$ defined in \eqref{eq:blow-up_map} is applied directly to the extended PDE system \eqref{eq:main_system_extended}, leading to PDE systems in the entry, rescaling and exit charts $K_1$, $K_2$ and $K_3$ respectively after a suitable desingularisation in both time and space (we refer again to Remarks \ref{rem:desingularisation} and \ref{rem:desingularisation_2}).
	\item[(b)] The linearisation about the steady states corresponding to the extension of attracting critical manifolds up to the blow-up manifold features a spectral gap; we refer to Lemma \ref{lem:spectrum_K1} and Figure \ref{fig:spectral_gap}. This resolves the marginally stable situation observed in the blown-down problem with $\mu = \eps = 0$, which is shown in Figure \ref{fig:spectrum_static}. This is analogous to one of the most important benefits of geometric blow-up based approaches in the ODE setting, namely, that a loss of hyperbolicity can often be resolved within the blown-up space; we refer again to Remark \ref{rem:resolution} for details.
	\item[(c)] The spectral gap property described in (b) allowed for the application of center manifold theory for evolution equations in Banach spaces. We applied these results, as they are formulated in \cite{Haragus2010}, in order to identify $2$-dimensional, strongly attracting local center manifolds which in turn allowed us to provide a rigorous and geometric description of the local dynamics within the entry and exit charts $K_1$ and $K_3$.
	\item[(d)] The extension of the local center manifold $\mathcal M_1^a$ in chart $K_1$ into and through the rescaling chart $K_2$ is described by Lemma \ref{lem:K2_connection_pitchfork} in the pitchfork case, and by Lemma \ref{lem:K2_connection_transcritical} in the transcritical case. In both cases, the details of this extension follow directly from the analysis of the corresponding ODE analysis in \cite{Krupa2001c}, because of the fact that $\mathcal M_1^a$ (and therefore also its extension) only contains spatially independent solutions.
	\item[(e)] We showed that PDE solutions in the rescaling chart $K_2$ `track' the extended manifold described above in (d), using direct a priori arguments to provide a bound on an error function which ensures that solutions stay within a tubular neighbourhood for finite time intervals; this is described in Propositions \ref{prop:K2_dynamics} and \ref{prop:K2_dynamics_transcritical}. Here it is also worthy to emphasise that the proofs of Propositions \ref{prop:K2_dynamics} and \ref{prop:K2_dynamics_transcritical} are inspired by established and relatively constructive approaches to the validation of modulation equations; we refer again to \cite{Schneider2017} and the many references therein.
\end{itemize}
With the possible exception of (d), the features raised in (a)-(e) above do not appear to rely in any critical way on the specific structure of the equations (although the precise manner in which a spectral degeneracy may be resolved is of course expected to vary on a case-by-case basis).

\

Looking forward, there are a number of analytical obstacles which remain to be overcome. The question of precisely how to incorporate $u$ into the geometry of the blown-up space in situations where a `fast escape' occurs (e.g.~near a fold, or in the transcritical case considered in this work with $\lambda < 0$) is still open; we refer again to Remark \ref{rem:spherical_blow-up}. The question of whether canard-type dynamics for $\lambda = \mathcal O(\eps)$ can be understood using geometric blow-up is also open. Of course, one could also allow for more complicated slow dynamics, different singularities, or more complicated background states. These and other related questions remain for future work.

\section*{Acknowledgements}

This work is an extension of preliminary work in the MSc thesis \cite{Martinez-Sanchez2024}. SJ would like to thank Peter Szmolyan for helpful discussions, and in particular for suggestions which led to the treatment of the dynamics within the rescaling chart(s) $K_2$. CK would like to thank the VolkswagenStiftung for support via a Lichtenberg Professorship. AM was supported by the European Union under the ERC Consolidator Grant No 101123223 (SSNSD), and by the AEI project PID2021-125021NAI00 funded by MICIU/AEI/10.13039/501100011033 and by FEDER (Spain).









\appendix

\section*{Appendix}

The following appendices provide background and additional material referred to within the main body of the paper. Appendix \ref{app:wieghted_spaces} is devoted to functional-analytic properties of the weighted spaces $\widetilde{\mathcal Y}$ and $\widetilde{\mathcal Z}$, Appendix \ref{app:center_manifold_theory} contains the hypotheses that need to be checked in order to apply the formulation of center manifold theory in \cite{Haragus2010}, and Appendix \ref{app:bounds} is devoted to the proof of Lemma \ref{lem:resolvant_bound}, which is necessary for the proof of Lemma \ref{lem:resolvent_K1}.

\section{Functional-analytic properties of the weighted spaces $\widetilde{\mathcal Y}$ and $\widetilde{\mathcal Z}$}
\label{app:wieghted_spaces}

In the following we state and prove a number of basic properties of the polynomially weighted spaces $\widetilde{\mathcal Y}$ and $\widetilde{\mathcal Z}$ defined in \eqref{eq:tilde_Y} and \eqref{eq:tilde_Z} respectively, which are used extensively throughout the manuscript.

\begin{lemma}
	\label{lem:density}
	The weighted spaces $\widetilde{\mathcal Y}$ and $\widetilde{\mathcal Z}$ are dense in $\widetilde{\mathcal X}$.
\end{lemma}

\begin{proof}
	Let $\widetilde{\mathcal W}$ be either $\widetilde{\mathcal Y}$ or $\widetilde{\mathcal Z}$. We show that for any $u \in \widetilde{\mathcal X}$ and $\epsilon > 0$, there exists a function $u \in \widetilde{\mathcal W}$ such that $\| u - v \|_\infty < \epsilon$. The constant $\epsilon$ should be distinguished from the small parameter $\eps$ used throughout the main body of the paper.
	
	Fix $u \in \widetilde{\mathcal X}$ and an arbitrarily small $\epsilon > 0$. Since $\widetilde{\mathcal X} \subset \mathcal C^0_{\textup{b,unif}}(\R)$ and $\mathcal C^2_{\textup{b,unif}}(\R)$ is dense in $\mathcal C^0_{\textup{b,unif}}(\R)$, there exists a function $\tilde v \in \mathcal C^2_{\textup{b,unif}}(\R)$ such that $\| u - \tilde v \|_\infty < \epsilon$. The idea is to define $v := \phi_M \tilde v$, where $\phi_M$ is a $C^\infty$ `cut-off-like' function which is defined so that $v(x)$ satisfies
	\[
	v(x) = 
	\begin{cases}
		u_{-\infty} , & x < - M - 1, \\
		\tilde v(x) , & x \in (-M, M), \\
		u_{+\infty} , & x > M + 1,
	\end{cases}
	\]
	where $u_{\pm \infty} := \lim_{x \to \pm \infty} u(x)$. Notice that $v \in \widetilde{\mathcal W}$. Choosing $M$ sufficiently large and assuming that $\phi_M$ has been suitably defined for $|x| \in [M, M+1]$ ensures that $\| u - v \|_\infty < \epsilon$, as required.
\end{proof}
%

\begin{lemma}
	\label{lem:algebra}
	$\widetilde{\mathcal Y}$ and $\widetilde{\mathcal Z}$ are algebras, i.e.~if $u, v \in \widetilde{\mathcal Y}$ ($u, v \in \widetilde{\mathcal Z}$), then $u v \in \widetilde{\mathcal Y}$ ($u v \in \widetilde{\mathcal Z}$).
\end{lemma}

\begin{proof}
	We prove that $\widetilde{\mathcal Z}$ is an algebra; the proof that $\widetilde{\mathcal Y}$ is an algebra is similar and simpler. Let $u, v \in \widetilde{\mathcal Z}$. In order to show that $u v \in \widetilde{\mathcal Z}$, we need to show that $\|u v\|_{\widetilde{\mathcal Z}} < \infty$. Direct calculations involving a product rule and several triangle inequalities leads to
	\[
	\begin{split}
		\| u v \|_{\widetilde{\mathcal Z}} =&\ \| u v \|_\infty + \| (1 + |\cdot| ) (u v)' \|_\infty + \| (1 + (\cdot)^2) (uv)'' \|_\infty \\
		\leq&\ \| u \|_\infty \| v \|_\infty + \| u \|_\infty \| (1 + |\cdot|) v' \|_\infty + \| v \|_\infty \| (1 + |\cdot|) u' \|_\infty + \| u \|_\infty \| (1 + (\cdot)^2) v'' \|_\infty \\
		&+ 2 \| (1 + (\cdot)^2 ) u' v' \|_\infty + \| v \|_\infty \| (1 + (\cdot)^2) u'' \|_\infty .
	\end{split}
	\]
	The first $5$ terms appearing in the right-hand side of the inequality above can bounded from above by $\| u \|_{\widetilde{\mathcal Z}} \| v \|_{\widetilde{\mathcal Z}}$; this follows from direct calculations if one expands the $\| u \|_{\widetilde{\mathcal Z}} \| v \|_{\widetilde{\mathcal Z}}$ using the definition of the $\widetilde{\mathcal Z}$-norm. Since $1 + x^2 \leq (1 + |x|)^2$ for all $x \in \R$, we can also bound the last term using
	\[
		\| (1 + (\cdot)^2 ) u' v' \|_\infty \leq 
		\| (1 + |\cdot|) u' \|_\infty \| (1 + |\cdot|) v' \|_\infty \\ 
		\leq \| u \|_{\widetilde{\mathcal Z}} \| v \|_{\widetilde{\mathcal Z}} .
	\]
	It follows that
	\[
	\| u v \|_{\widetilde{\mathcal Z}} \leq 3 \| u \|_{\widetilde{\mathcal Z}} \| v \|_{\widetilde{\mathcal Z}} ,
	\]
	implying that $u v \in \widetilde{\mathcal Z}$, as required.
\end{proof}

\begin{lemma}
	\label{lem:norm_equivalence}
	Constant spatial scalings yield equivalent norms in the following sense: for any $u \in \widetilde{\mathcal Y}$ ($u \in \widetilde{\mathcal Z}$) and for any fixed $\alpha > 0$, let $x = \alpha \tilde x$ and $u(x) = \tilde u(\tilde x)$. Then $\| u(\cdot) \|_{\widetilde{\mathcal Y}} \sim \| \tilde u(\cdot) \|_{\widetilde{\mathcal Y}}$ ($\| u(\cdot) \|_{\widetilde{\mathcal Z}} \sim \| \tilde u(\cdot) \|_{\widetilde{\mathcal Z}}$).
\end{lemma}

\begin{proof}
	As in the proof of Lemma \ref{lem:algebra} above, we only present a proof in $\widetilde{\mathcal Z}$ because the proof in $\widetilde{\mathcal Y}$ is similar and simpler. Fix $\alpha > 0$. We need to show that there are positive constants $c_+ \geq c_- \geq 0$ such that
	\begin{equation}
		\label{eq:norm_equivalence}
		c_- \| \tilde u(\cdot) \|_{\widetilde{\mathcal Z}} \leq 
		\| u(\cdot) \|_{\widetilde{\mathcal Z}} \leq 
		c_+ \| \tilde u(\cdot) \|_{\widetilde{\mathcal Z}} .
	\end{equation}
	Since $\| u(\cdot) \|_\infty = \| \tilde u(\cdot) \|_\infty$ follows directly from the definition of the supremum norm, it remains to consider the terms involving derivatives. Writing $x = \alpha \tilde x$ and $u(x) = \tilde u(\tilde x)$ as in the statement of the lemma, we obtain
	\[
	\begin{split}
		| (1 + |x|) u'(x) | &= \left| (\alpha^{-1} + | \tilde x |) \tilde u'(\tilde x) \right| , \\ 
		| (1 + x^2) u''(x) | &= \left| (\alpha^{-2} + \tilde x^2) \tilde u''(\tilde x) \right| ,
	\end{split}
	\]
	where $\tilde u'(\tilde x) = \partial \tilde u / \partial \tilde x$ and $\tilde u''(\tilde x) = \partial^2 \tilde u / \partial \tilde x^2$. From here one can show directly that
	\[
	C_1^- | (1 + |\tilde x|) u'(\tilde x) | \leq 
	| (1 + |x|) u'(x) | \leq 
	C_1^+ | (1 + |\tilde x|) u'(\tilde x) | ,
	\]
	and
	\[
	C_2^- | (1 + \tilde x^2) u''(\tilde x) | \leq 
	| (1 + x^2) u''(x) | \leq 
	C_2^+ | (1 + \tilde x^2) u''(\tilde x) | ,
	\]
	where the ($\alpha$-dependent) constants are given by
	\[
	\begin{split}
		C_j^- &= C_j^-(\alpha) := \min\{1, \alpha^{-j}\} , \\
		C_j^+ &= C_j^+(\alpha) := \max\{1, \alpha^{-j}\} ,
	\end{split}
	\]
	for $j \in \{1, 2\}$. These inequalities imply that \eqref{eq:norm_equivalence} is satisfied with
	\[
	c_- = 2 \min\{C_1^-, C_2^- \} , \quad  
	c_+ = 2 \max\{C_1^-, C_2^- \} ,
	\]
	thereby concluding the proof.
\end{proof}

Lemma \ref{lem:norm_equivalence} is particularly useful when it comes to changing between local coordinates in charts $K_i$. In particular, in the pitchfork case $s = 3$, it allows us to choose the sections $\Sigma_1^{\textup{out}}$ and $\Sigma_2^{\textup{out}}$ so that $\kappa_{12}(\Sigma_1^{\textup{out}}) \subseteq \Sigma_2^{\textup{in}}$ and $\kappa_{23}(\Sigma_2^{\textup{out}}) \subseteq \Sigma_3^{\textup{in}}$, since solutions in $\Sigma_1^{\textup{out}}$ satisfy
\begin{equation}
	\label{eq:x_scaling1}
	x_1 = r_1 x = \frac{x_2}{\Omega^{1/2}} = \delta^{1/4} x_2 ,
\end{equation}
and solutions in $\Sigma_2^{\textup{out}}$ satisfy
\begin{equation}
	\label{eq:x_scaling2}
	x_2 = r_2 x = \delta^{1/4} x_3 ,
\end{equation}
where we used the change of coordinates formulae in \eqref{eq:kappa_ij} when $s = 3$, together with the definitions of the sections. Since \eqref{eq:x_scaling1} and \eqref{eq:x_scaling2} are constant scalings, Lemma \ref{lem:norm_equivalence} implies that the norms are equivalent upon changing coordinates. Similar arguments show that $\Sigma_1^{\textup{out}}$ and $\Sigma_2^{\textup{out}}$ can be chosen such that $\kappa_{12}(\Sigma_1^{\textup{out}}) \subseteq \Sigma_2^{\textup{in}}$ and $\kappa_{23}(\Sigma_2^{\textup{out}}) \subseteq \Sigma_3^{\textup{in}}$ in the transcritical case $s = 2$.


\

We conclude this section with a bound on the heat semigroup generated by $\partial_{x_2}^2$ in $\widetilde{\mathcal Z}$, which is helpful in proving Propositions \ref{prop:K2_dynamics} and \ref{prop:K2_dynamics_transcritical}.

\begin{lemma}
	\label{lem:semigroup}
	Let $(\e^{t_2 \Lambda_2})_{t_2 \geq 0}$ denote the heat semigroup generated by $\partial_{x_2}^2$. Then 
	\begin{equation}
		\label{eq:operator_norm_2}
		\| \e^{t_2 \Lambda_2} \|_{\widetilde{Z} \to \widetilde{Z}} \leq 1 
	\end{equation}
	for all $t_2 > 0$.
\end{lemma}

\begin{proof}
	This can be shown directly using the solution formula
	\[
	u(x,t) = \e^{t \Lambda} u(x, 0) = 
	\frac{1}{\sqrt{4 \pi t}} \int_{-\infty}^\infty \e^{-(x - y)^2 / 4t} u(y,0) \dd y ,
	\]
	where here and in the remainder of the proof we drop the subscript ``2'' for notational simplicity. A direct application of a triangle inequality yields
	\[
	\| u(\cdot, t) \|_\infty \leq \| u(\cdot, 0) \|_\infty ,
	\]
	and integrating by parts leads to
	\[
	\partial_x u(x,t) = \e^{t \Lambda} \partial_x u(x, 0) , \qquad
	\partial_x^2 u(x,t) = \e^{t \Lambda} \partial_x^2 u(x, 0) ,
	\]
	for each $t > 0$. Using the left-most equation above, we obtain
	\[
		\| (1 + |\cdot| ) \partial_x u(\cdot,t) \|_\infty \leq 
		\| (1 + |\cdot| ) \partial_x u(\cdot,0) \|_\infty 
	\]
	and, similarly, the right-most equations leads to
	\[
	\| (1 + (\cdot)^2 ) \partial_x^2 u(\cdot,t) \|_\infty \leq 
	\| (1 + (\cdot)^2 ) \partial_x^2 u(\cdot,0) \|_\infty .
	\]
	for all $t > 0$. It follows that
	\[
	\| u(\cdot, t) \|_{\widetilde{Z}} \leq \| u(\cdot, 0) \|_{\widetilde{Z}} ,
	\]
	for all $t > 0$, thereby proving \eqref{eq:operator_norm_2} and the desired result.
\end{proof}

\section{Hypotheses for the center manifold theorem}
\label{app:center_manifold_theory}

The results in \cite{Haragus2010} apply to evolution equations on Banach spaces of the form
\[
\frac{\dd \bm{u}}{\dd t} = \Lop \bm{u} + \Nop (\bm{u}) ,
\]
where $\Lop$ is a closed, densely defined linear operator, and $\Nop$ is a nonlinear operator. In this paper, we work with the particular spaces $\mathcal X \subseteq \mathcal Y \subseteq \mathcal Z$ defined in \eqref{eq:spaces}, but the more general requirement in \cite{Haragus2010} is that $\mathcal Z \hookrightarrow \mathcal Y \hookrightarrow \mathcal X$ with continuous embedding (for the particular choice of spaces in this work, the embedding is naturally indued by the identity map).

In order to apply center manifold theorems from \cite{Haragus2010}, three main hypotheses need to be verified. These are formulated below in a notation which is consistent with the notation presented in the main body of the article.

The first hypothesis ensures that $\bm{u} = \bm{0}$ is a steady state, and that $\Lop$ is the linearisation. In order to state it, we let $\mathcal L(X,Y)$ denote the space of bounded/continuous linear operators from $X$ to $Y$.

\begin{hyp}
	\label{hyp:1}
	\cite[Hypothesis 2.1]{Haragus2010}
	The linear resp.~nonlinear operators $\Lop$ and $\Nop$ satisfy the following assumptions:
	\begin{enumerate}
		\item[(i)] $\Lop \in \mathcal L(\mathcal Z, \mathcal X)$;
		\item[(ii)] There exists an integer $k \geq 2$ and a neighbourhood $\mathcal V$ of $\bm{0}$ in $\widetilde{\mathcal Z}$ such that $N \in \mathcal C^k(\mathcal V, \widetilde{\mathcal Y})$, $\Nop(\bm{0}) = \bm{0}$ and $\D \Nop(\bm{0}) = \bm{0}$.
	\end{enumerate}
\end{hyp}

The second hypothesis is the well-known `spectral gap requirement', which applies to the spectrum $\sigma(\Lop)$. In order to state it succinctly, we write
\[
\sigma(\Lop) = \sigma_-(\Lop) \cup \sigma_0(\Lop) \cup \sigma_+(\Lop) ,
\]
where
\[
	\sigma_-(\Lop) = \left\{ \lambda \in \textup{Re} \lambda < 0 \right\} , \qquad
	\sigma_0(\Lop) = \left\{ \lambda \in \textup{Re} \lambda = 0 \right\} , \qquad
	\sigma_+(\Lop) = \left\{ \lambda \in \textup{Re} \lambda > 0 \right\} .
\]
The hypothesis can then be formulated as follows:

\begin{hyp}
	\label{hyp:2}
	\cite[Hypothesis 2.4]{Haragus2010}
	The spectrum $\sigma(\Lop)$ possesses a spectral gap, i.e.~there exists a constant $\gamma > 0$ such that
	\[
	\sup_{\lambda \in \sigma_-} \text{Re} (\lambda) < - \gamma, \qquad 
	\inf_{\lambda \in \sigma_+} \text{Re} (\lambda) > \gamma .
	\]
	Moreover, $\sigma_0$ contains only a finite number of eigenvalues, each of which has finite algebraic multiplicity.
\end{hyp}

Notice in particular that Hypothesis \ref{hyp:2} implies that any center subspace $\mathcal E_0$ is finite-dimensional.

The third and final main hypothesis pertains to the existence and uniqueness of solutions to the linearised problem which is obtained after projection onto the linear subspace which is complementary to the center eigenspace $\mathcal E_0$. In general settings, one must check the existence and uniqueness of solutions to this problem directly, i.e.~one has to verify \cite[Hypothesis 2.7]{Haragus2010}, which can be difficult in applications. Fortunately, a simplification holds in the case of semilinear equations \cite[Sec.~2.2.3]{Haragus2010}, for which this problem reduces to the problem of verifying a pair of resolvant estimates.

\begin{hyp}
	\label{hyp:3}
	\cite[Hypothesis 2.15]{Haragus2010}
	There exists constants $\omega_0 > 0$, $c > 0$ and $\alpha \in [0,1)$ such that for all $\omega \in \R$ satisfying $| \omega | \geq \omega_0$, $i \omega \in \textup{Res}(\Lop)$ and the following resolvent estimates are satisfied:
	\[
		\| (i \omega \text{Id} - \Lop) \|_{\mathcal L(\mathcal X)} \leq \frac{c}{| \omega |} , \qquad
		\| (i \omega \text{Id} - \Lop) \|_{\mathcal L(\mathcal Y, \mathcal Z)} \leq \frac{c}{| \omega |^{1 - \alpha}} ,
	\]
	where $\mathcal L(\mathcal X) := \mathcal L(\mathcal X, \mathcal X)$.
\end{hyp}

\section{Bounds for the proof of Lemma \ref{lem:resolvent_K1}}
\label{app:bounds}

The following result is used to prove the second resolvent bound in Lemma \ref{lem:resolvant_bound}. 

\begin{lemma}
	\label{lem:resolvant_bound}
	Let $a_1 \in \widetilde{\mathcal Z}$ be the function defined by \eqref{eq:a1}, $b_1 \in \widetilde{\mathcal Y}$, and $\omega \in \R$ be such that $| \omega | \geq 1$. There exists a constant $C \geq 0$ such that
	\begin{equation}
		\label{eq:bound1}
		\frac{\| (1 + (\cdot)^2 ) a_1'' \|_\infty}{\| (1 + | \cdot | )b_1' \|_\infty} \leq \frac{C}{| \omega |^{1/2}} .
	\end{equation}
\end{lemma}

\begin{proof}
	In order to simplify notation, we rewrite the formula for $a_1''(x_1)$ in \eqref{eq:a1} as
	\[
	a_1''(x_1) = \frac{\e^{\eta x_1}}{2} I_1(x_1) - \frac{\e^{- \eta x_1}}{2} I_2(x_1) , 
	\]
	where
	\[
	I_1(x_1) := \int_{x_1}^\infty \e^{-\eta \xi} b_1'(\xi) d \xi , \qquad 
	I_2(x_1) := \int_{-\infty}^{x_1} \e^{\eta \xi} b_1'(\xi) d \xi .
	\]
	To prove the bound in \eqref{eq:bound1}, we shall consider the corresponding bound for $|x_1| \leq M$ and $|x_1| > M$ separately, where $M > 0$ is a constant which shall be chosen to be sufficiently large for the validity of asymptotic arguments as $| x_1 | \to \infty$ in the latter case. It is straightforward to obtain a bound corresponding to \eqref{eq:bound1} on the bounded set $\{|x_1| \leq M\}$, since
	\begin{equation}
		\label{eq:finite_bound}
			\frac{\sup_{x_1 \in [-M,M]} (1 + x_1^2) | a_1''(x_1) |}{\sup_{x_1 \in [-M,M]} (1 + |x_1|) | b_1'(x_1) |}
			\leq (1 + M^2) \frac{\sup_{x_1 \in [-M,M]} | a_1''(x_1) |}{\sup_{x_1 \in [-M,M]} | b_1'(x_1) |}
			\leq (1 + M^2) \frac{1}{\eta_r}
	\end{equation}
	where $\eta_r = \text{Re} (\eta)$ and the last inequality follows by a direct calculation using the expression for $a_1''(x_1)$ on $\{|x_1| \leq M\}$ and a pair of triangle inequalities.
	
	It remains to prove a similar bound for the far field, i.e.~for $|x_1| > M$. We start by deriving bounds for the integrals $I_1(x_1)$ and $I_2(x_1)$ which are valid when $|x_1| > M$ (assuming $M$ is chosen large enough). Consider first the integral $I_1(x_1)$ as $x_1 \to \infty$, which can be rewritten as
	\[
	I_1(x_1) = \int_1^{\infty} \e^{- \eta x_1 s} \tilde b_1'(s) \dd s
	\]
	after making the substitution $s = \xi / x_1$ and writing $\tilde b_1'(s) := b_1'(\xi / x_1)$. Notice that the supremum norm of $\tilde b_1'(s)$ and $b_1'(\xi / x_1)$ coincide on $\{ x_1 > M \}$, and that the lower integration bound is constant now. It follows that
	\begin{equation}
		\label{eq:I1_bound}
			\sup_{x_1 > M} I_1(x_1) \leq 
			\sup_{x_1 > M} |b_1'(x_1)|\int_1^\infty \e^{ - \eta_r x_1 s} \dd s
			= \e^{- \eta_r x_1} \frac{\sup_{x_1 > M} |b_1'(x_1)|}{\eta_r x_1} .
	\end{equation}
	The integral $I_2(x_1)$ can be rewritten as
	\[
	I_2(x_1) = \int_{-\infty}^0 \e^{\eta \xi} b'(\xi) \dd \xi + \int_0^1 \e^{\eta x_1 s} \tilde b_1'(s) \dd s .
	\]
	The first integral can be bounded directly using
	\begin{equation}
		\label{eq:integral_bound}
		\sup_{x_1 > M} \left| \int_{-\infty}^0 \e^{\eta \xi} b'(\xi) \dd \xi \right| \leq \frac{\sup_{x_1 > M} | b_1'(x_1) |}{\eta_r} ,
	\end{equation}
	and the second, which we denote by $\tilde I_2(x_1)$, can be bounded using arguments similar to those used in the derivation of a bound for $I_1(x_1)$ in \eqref{eq:I1_bound} above. We obtain
	\[
	\sup_{x_1 > M} \tilde I_2(x_1) \leq 
	\e^{\eta_r x_1} \frac{\sup_{x_1 > M} |b_1'(x_1)|}{\eta_r x_1} .
	\]
	Taking the supremum norm of $a_1''$, applying a triangle inequality and substituting the bounds obtained above leads to
	\[
	\sup_{x_1 > M} a_1''(x_1) \leq C \frac{\sup_{x_1 > M} |b_1'(x_1)|}{\eta_r x_1} ,
	\]
	for a constant $C > 1$ (which we introduce in order to absorb the exponentially small contribution due to \eqref{eq:integral_bound}). It follows that
	\[
	\frac{(1 + x_1^2) | a_1''(x_1) |}{(1 + x_1) |b_1' (x_1)|} \leq
	C \frac{1 + x_1^2}{x_1 (1 + x_1)} \frac{1}{\eta_r} \leq 
	\frac{\tilde C}{\eta_r} ,
	\]
	where $\tilde C = C + c$ for some small but fixed constant $c > 0$ (whose size depends on $M$), for all $x_1 > M$. This implies that
	\[
	\sup_{x_1 > M} (1 + x_1^2) | a_1''(x_1) | \leq 
	\frac{\tilde C}{\eta_r} \sup_{x_1 > M} (1 + x_1) |b_1' (x_1)| ,
	\]
	and therefore that
	\[
	\frac{\sup_{x_1 > M} (1 + x_1^2) | a_1''(x_1) |}{\sup_{x_1 > M} (1 - x_1) |b_1' (x_1)|} \leq \frac{\tilde C}{\eta_r} .
	\]

	Similar arguments -- which we omit here for brevity -- can be used to derive the analogous bound on $\{ x_1 < -M \}$, i.e.~
	\[
	\sup_{x_1 < M} a_1''(x_1) \leq C \frac{\sup_{x_1 < M} |b_1'(x_1)|}{\eta_r |x_1|} ,
	\]
	leading in turn to
	\[
	\frac{\sup_{x_1 < M} (1 + x_1^2) | a_1''(x_1) |}{\sup_{x_1 < M} (1 - x_1) |b_1' (x_1)|} \leq \frac{\tilde C}{\eta_r} .
	\]
	
	Combining these bounds with the bound obtained on $[-M,M]$, recall \eqref{eq:finite_bound}, it follows that
	\[
	\frac{\| (1 + (\cdot)^2 ) a_1'' \|_\infty}{\| (1 + | \cdot | )b_1' \|_\infty} \leq \frac{1 + M^2}{\eta_r} .
	\]
	Solving $\eta = \sqrt{1 + i \omega}$ for $\eta_r$ leads to $\eta_r = | \omega |^{1/2} / \sqrt{2}$. Substituting this into the above yields the bound \eqref{eq:bound1}, thereby completing the proof.
\end{proof}

\bibliographystyle{siam}

\bibliography{slow_passage}



\end{document}